\documentclass[12pt,reqno, final]{amsart}

\usepackage{amssymb,amsmath,graphicx,
	amsfonts,euscript}
\usepackage{color}
\usepackage{cite}
\allowdisplaybreaks

\newcommand{\be}{\begin{eqnarray}}
\newcommand{\ee}{\end{eqnarray}}
\newcommand{\bes}{\begin{eqnarray*}}
\newcommand{\ees}{\end{eqnarray*}}

\newcommand{\om}{\omega}

\newcommand{\na}{\nabla}

\newcommand{\p}{\partial}

\newtheorem{thm}{Theorem}[section]
\newtheorem{pro}{Proposition}[section]
\newtheorem{lem}{Lemma}[section]

\newcommand{\beq}{\begin{equation}}
\newcommand{\eeq}{\end{equation}}
\newcommand{\ben}{\begin{eqnarray}}
\newcommand{\een}{\end{eqnarray}}
\newcommand{\beno}{\begin{eqnarray*}}
\newcommand{\eeno}{\end{eqnarray*}}

\numberwithin{equation}{section}

\begin{document}
	
	\title[Stability problem on the 2D Boussinesq equations]{The stabilizing effect of the  temperature on buoyancy-driven fluids }
	
	\author[O. Ben Said, U. Pandey and J.  Wu ]{Oussama Ben Said$^{1}$,  Uddhaba Raj Pandey $^{2}$ and Jiahong Wu$^{3}$}
	
	\address{$^1$ Department of Mathematics, Oklahoma State University, Stillwater, OK 74078, USA}
	
	\email{obensai@ostatemail.okstate.edu}

	\address{$^{2}$ Department of Mathematics, Oklahoma State University, Stillwater, OK 74078, USA}
	
	\email{uddhaba@okstate.edu}
	
	\address{$^3$ Department of Mathematics, Oklahoma State University, Stillwater, OK 74078, USA}
	
	\email{jiahong.wu@okstate.edu}

	\vskip .2in
	\begin{abstract}
	The Boussinesq system for buoyancy driven fluids couples the momentum 
	equation forced by the buoyancy with the convection-diffusion equation for the temperature. 
	One fundamental issue on the Boussinesq system is the stability problem on perturbations
	near the hydrostatic balance. This problem can be extremely difficult when the system lacks
	full dissipation. This paper solves the stability problem for a two-dimensional 
	Boussinesq system with only vertical dissipation and horizontal thermal diffusion. We 
	establish the stability for the nonlinear system and derive precise large-time behavior for 
	the linearized system. The results presented in this paper reveal a remarkable phenomenon
	for buoyancy driven fluids. That is, the temperature actually smooths and stabilizes the
	fluids. If the temperature were not present, the fluid is governed by the 2D Navier-Stokes
	with only vertical dissipation and its stability remains open. It is the coupling and interaction between the temperature and the velocity in the Boussinesq system that makes
	the stability problem studied here possible. Mathematically the system can be reduced to degenerate and 
	damped wave equations that fuel the stabilization.	
	\end{abstract}
	\maketitle

\section{Introduction}

This paper intends to reveal and rigorously prove the fact that the temperature can actually have
a stabilizing effect on the buoyancy-driven fluids. As we know, buoyancy driven flows such as geophysical fluids and various Rayleigh-B\'{e}nard convection are modeled by the Boussinesq equations. Our study is based on the following special two-dimensional (2D) Boussinesq system
with partial dissipation 
\beq\label{bb}
\begin{cases}
\partial_t U + U\cdot \nabla U= -\nabla {P}+ \nu\, \partial_{22}U + \Theta {\mathbf e}_2, \quad x\in\mathbb R^2, \,\,t>0, \\
\partial_t \Theta + U \cdot \nabla \Theta = \eta\, \partial_{11}\Theta, \\
\nabla \cdot U=0, 
\end{cases}
\eeq
where $U$ denotes the fluid velocity, $P$ the pressure, $\Theta$ the temperature, $\nu>0$ the 
kinematic viscosity, and $\eta$ the thermal diffusivity. Here ${\mathbf e}_2$ is the unit vector in the vertical direction. The dissipation in the velocity equation is anisotropic and is only 
in the vertical direction. The partial differential equations (PDEs) with only degenrate dissipation are relevant in certain physical regimes, and one of the most notable examples is Prandtl's equation. Another reason for including only partial dissipation in the velocity equation is to help better reveal the smoothing and stabilization effect of the temperature. More precise explanation will be presented later. 

\vskip .1in 
The Boussinesq equations for buoyancy driven fluids are widely used in the modeling and 
study of atmospheric and oceanographic flows and the Rayleigh-B\'{e}nard convection
(see, e.g., \cite{ConD, DoeringG,Maj,Pe}). The Boussinesq equations are also mathematically
important. The 2D Boussinesq equations serve as a lower dimensional model of the 3D hydrodynamics equations. In fact, the 2D Boussinesq equations retain some key features of the 3D Euler and Navier-Stokes equations such as the vortex stretching mechanism.  The inviscid 2D Boussinesq equations can be identified as the Euler equations for the 3D axisymmetric swirling flows \cite{MaBe}. Fundamental issues on the Boussinesq systems such as the global well-posedness problem have attracted a lot of interests recently, especially when the systems involve only
partial dissipation or no dissipation at all (see, e.g., \cite{ACSW,ACW10,ACW11,ACWX,Nicki,CaoWu1,Ch,ChConWu, Chae-Nam1997,ChaeWu,CKY,ConVWu,Dai,DP2,DP3,ElJ,Tarek2,LBHe,HKR1,HKR2,HL,Hu,
	HuWangWu,JMWZ,JWYang,KRTW,KiTan,LaiPan,LLT,LSWu,LiTiti,MX,SaWu,SteWu,Wan,Wen,Wu,
	Wu_Xu,Wu_Xu_Xue_Ye,Wu_Xu_Ye,Xu0,YangJW,YangJW2,YeXu,Zhao2010}). The study on the stability of several  steady states to the Boussinesq system has recently gained momentum due to their physical applications. More details will be described in the later part of the introduction.

\vskip .1in 
The main purpose of this paper is to understand the stability and large-time behavior of 
perturbations near the hydrostatic equilibrium $(U_{he}, \Theta_{he})$ with
$$
U_{he} =0, \quad \Theta_{he} =x_2.
$$
For the static velocity $U_{he}$, the momentum equation is satisfied when the pressure 
gradient is balanced by the buoyancy force, namely 
$$
-\na P_{he} + \Theta_{he} \, \mathbf e_2 =0 \quad\mbox{or}\quad P_{he} = \frac12 x_2^2. 
$$
$(U_{he}, P_{he}, \Theta_{he})$ is a very special steady solution with great physical significance. 
In fact, our atmosphere is mostly in hydrostatic equilibrium with the upward pressure gradient 
force balanced by the buoyancy due to the gravity. 

\vskip .1in 
To understand the desired stability, we write the equation of the perturbation denoted by $(u, p, \theta)$, where
$$
u = U -U_{he}, \quad p= P- P_{he} \quad\mbox{and} \quad \theta = \Theta- \Theta_{he}. 
$$
It follows easily from (\ref{bb}) that the perturbation $(u, p, \theta)$ satisfies 
\beq\label{bb1}
    \begin{cases}
        \partial_t u+u\cdot \nabla u= -\nabla p+ \nu\, \partial_{22}u+ \theta \mathbf e_2, \\
        \partial_t \theta + u \cdot \nabla \theta + u_2= \eta\, \partial_{11} \theta, \\
        \nabla \cdot u=0, \\
        u(x,0) =u_0(x), \quad \theta(x,0) =\theta_0(x).    
    \end{cases}
\eeq
The only difference between (\ref{bb}) and (\ref{bb1}) is an extra term $u_2$ 
(the vertical component of $u$) in (\ref{bb1}), 
which plays a very important role in balancing the energy. In order to assess the stability, we need to establish that the solution $(u, b)$ of (\ref{bb1}) corresponding to any sufficiently small initial perturbation $(u_0, b_0)$ (measured in the Sobolev norm $H^2(\mathbb R^2)$) remains small for all time. This does not appear to be an easy task when there is only vertical 
velocity dissipation and horizontal thermal diffusion. 

\vskip .1in 
The lack of horizontal dissipation makes it hard to control the growth of the vorticity $\om = \na\times u$, which satisfies
\beq\label{vor}
\p_t \om + u\cdot\na \om = \nu\, \p_{22} \om + \p_1 \theta, \quad x\in \mathbb R^2, \, \,t>0.
\eeq
We can obtain a uniform bound on the $L^2$-norm of the vorticity $\om$ itself, but it does not 
appear possible to control the $L^2$-norm of the gradient of the vorticity, $\na\om$. If $\theta$ were identically zero, (\ref{vor}) becomes the 2D Navier-Stokes equation with degenerate dissipation, 
\beq\label{vor0}
\p_t \om + u\cdot\na \om = \nu\, \p_{22} \om, \quad x\in \mathbb R^2, \, \,t>0.
\eeq
(\ref{vor0}) always has a unique global solution $\om$ for any initial data $\om_0\in H^1(\mathbb R^2)$, but the issue of whether or not $\|\na \om(t)\|_{L^2}$ grows or decays as a function of $t$  remains an open problem. 
When $\nu=0$, (\ref{vor0}) becomes the 2D Euler vorticity equation 
$$
\p_t \om + u\cdot\na \om =0, \quad x\in \mathbb R^2, \, \,t>0.
$$
As demonstrated in several beautiful work (see, e.g., \cite{Den, Kis, Zla}), $\na\om(t)$ can grow even double exponentially in time. In contrast, solutions to the 2D Navier-Stokes equations with full dissipation 
$$
\p_t \om + u\cdot\na \om = \nu\, \Delta \om, \quad x\in \mathbb R^2, \, \,t>0
$$
have been shown to always decay in time (see, e.g., \cite{Sch0,Sch}). The lack of the horizontal dissipation in (\ref{vor0}) prevents us from mimicking the approach designed for the fully dissipative 
Navier-Stokes equations. In fact, when we estimate $\|\na \om(t)\|_{L^2}$, the issue is how to 
proceed from the energy equality 
$$
\frac12\,\frac{d}{dt} \|\na \om(t)\|_{L^2}^2 +  \nu \|\p_2 \na \om(t)\|_{L^2}^2 = - \int \na \om \cdot\na u\cdot \na \om\, dx. 
$$
It appears impossible to control the term on the right. In order to make use of the anisotropic 
dissipation, we can further decompose the nonlinearity into four component terms, 
\ben
\int \na \om \cdot\na u\cdot \na \om\, dx &=& \int \p_1 u_1 \, (\p_1 \om)^2\,dx + \int \p_1 u_2\, \p_1 \om\, \p_2 \om\,dx \label{hard}\\
&& + \int \p_2 u_1 \, \p_1 \om\, \p_2 \om\,dx + \int \p_2 u_2\, (\p_2\om)^2\,dx. \notag
\een 
However, the first two terms in (\ref{hard}) do not appear to admit suitable bounds due to the 
lack of control on the horizontal derivatives in the dissipation. Whether or not $\|\na \om(t)\|_{L^2}$ grows or decays in time remains an open problem.

\vskip .1in 
When we deal with the stability problem on (\ref{bb1}), we encounter exactly the same term in 
(\ref{hard}). How would it be possible to deal with the same difficulty when we now have a more complex system like (\ref{bb1})? It is the smoothing and stabilization effect of the temperature through 
the coupling and interaction that makes the stability problem on (\ref{bb1}) possible. We give a
quick explanation on this mechanism. Since the linear portion of the nonlinear system in
 (\ref{bb1}) plays a crucial role in the stability properties, we first eliminate the 
pressure term in (\ref{bb1}) to separate the linear terms from the nonlinear ones. Applying 
the Helmholtz-Leray projection $\mathbb P = I - \na\Delta^{-1}\na\cdot$ to the velocity equation yields
\beq\label{j8}
\p_t u = \nu \p_{22} u + \mathbb P (\theta \mathbf e_2) -\mathbb P(u\cdot\na u).
\eeq
By the definition of $\mathbb P$, 
\beq\label{j9}
\mathbb P (\theta \mathbf e_2) = \theta \mathbf e_2 - \na \Delta^{-1} \na\cdot (\theta \mathbf e_2) = \left[ \begin{array}{c} -\p_1 \p_2 \Delta^{-1} \theta \\ \theta - \p_2^2 \Delta^{-1} \theta\end{array}\right].
\eeq
Inserting (\ref{j9}) in (\ref{j8}) and writing (\ref{j8}) in terms of its component equations, we obtain 
\beq\label{newu}
\begin{cases}
	\p_t u_1 = \nu \, \p_{22} u_1 - \p_1\p_2 \Delta^{-1} \theta + N_1, \\
	\p_t u_2 = \nu \, \p_{22} u_2 + \p_1\p_1 \Delta^{-1} \theta + N_2,
\end{cases}
\eeq
where $N_1$ and $N_2$ are the nonlinear terms, 
$$
N_1 = -( u\cdot\na u_1 -\p_1 \Delta^{-1} \na\cdot(u\cdot\na u)), \quad 
N_2 = -( u\cdot\na u_2 -\p_2 \Delta^{-1} \na\cdot(u\cdot\na u)).
$$
By differentiating the first equation of (\ref{newu}) in $t$ yields 
$$
\p_{tt} u_1 = \nu \p_{22} \p_t u_1- \p_1\p_2 \Delta^{-1} \p_t \theta + \p_t N_1.
$$
Replacing $\p_t \theta$ by the equation of $\theta$, namely $\p_t \theta = \eta \, \p_{11} \theta - u_2 - u\cdot\na \theta$ gives 
$$
\p_{tt} u_1 = \nu \p_{22} \p_t u_1 + \p_1\p_2 \Delta^{-1} u_2 -\eta \, \p_{11} \p_1\p_2\Delta^{-1} \theta + \p_1 \p_2 \Delta^{-1}(u\cdot\na \theta) + \p_t N_1.
$$
By further replacing $\p_1\p_2\Delta^{-1} \theta$ by the first equation of (\ref{newu}), namely
$$
- \p_1\p_2 \Delta^{-1} \theta = \p_t u_1 -\nu \,\p_{22} u_1-N_1,
$$
we obtain 
\beno
\p_{tt} u_1 &=& \nu \p_{22} \p_t u_1 + \p_1\p_2 \Delta^{-1} u_2 + \eta \, \p_{11} (\p_t u_1 -\nu \,\p_{22} u_1-N_1)\\
&& + \,\p_1 \p_2 \Delta^{-1}(u\cdot\na \theta) + \p_t N_1,
\eeno
which leads to, due to the divergence-free condition $\p_2 u_2 = -\p_1 u_1$, 
\beq\label{u1t}
\p_{tt} u_1 -( \eta \p_{11} + \nu \p_{22})  \p_t u_1 + \nu \eta \p_{11} \p_{22} u_1 + \p_{11}\Delta^{-1} u_1 = N_3.
\eeq
Here $N_3$ contains the nonlinear terms, 
$$
N_3=(\p_t  -\eta \p_{11}) N_1 + \p_1 \p_2 \Delta^{-1}(u\cdot\na \theta).
$$
Through a similar process, $u_2$ and $\theta$ can be shown to satisfy 
\ben
&& \p_{tt} u_2 -( \eta \p_{11} + \nu \p_{22})  \p_t u_2 + \nu \eta \p_{11} \p_{22} u_2 + \p_{11}\Delta^{-1} u_2 = N_4, \label{u2t}\\
&& \p_{tt} \theta -( \eta \p_{11} + \nu \p_{22})  \p_t \theta + \nu \eta \p_{11} \p_{22} \theta + \p_{11}\Delta^{-1} \theta = N_5 \notag
\een
with 
\beno
&& N_4 = (\p_t -\eta \p_{11}) N_2 - \p_1 \p_1 \Delta^{-1}(u\cdot\na \theta),\\
&& N_5 = -(\p_t-\nu \p_{22})(u\cdot\na \theta) -N_2.
\eeno 
Combining (\ref{u1t}) and (\ref{u2t}) and rewriting them into the velocity vector form, we have 
converted (\ref{bb1}) into the following new system 
\beq\label{bb2}
\begin{cases}
	\p_{tt} u -( \eta \p_{11} + \nu \p_{22})  \p_t u + \nu \eta \p_{11} \p_{22} u + \p_{11}\Delta^{-1} u = N_6,\\
	\p_{tt} \theta -( \eta \p_{11} + \nu \p_{22})  \p_t \theta + \nu \eta \p_{11} \p_{22} \theta + \p_{11}\Delta^{-1} \theta = N_5,
\end{cases}
\eeq
where 
$$
N_6=(N_3, N_4) = -(\p_t -\eta \p_{11})\mathbb P(u\cdot\na u) + \na^\perp \p_1 \Delta^{-1}(u\cdot\na \theta)
$$
with $\na^\perp =(\p_2, -\p_1)$. By taking the curl of the velocity equation, we can also convert 
(\ref{bb2}) into a system of $\om$ and $\theta$, 
$$
\begin{cases}
	\p_{tt} \om -( \eta \p_{11} + \nu \p_{22})  \p_t \om + \nu \eta \p_{11} \p_{22} \om + \p_{11}\Delta^{-1} \om = N_7,\\
	\p_{tt} \theta -( \eta \p_{11} + \nu \p_{22})  \p_t \theta + \nu \eta \p_{11} \p_{22} \theta + \p_{11}\Delta^{-1} \theta = N_5,
\end{cases}
$$
where 
$$
N_7 = - (\p_t -\eta \p_{11})(u\cdot\na \om) -\p_1 (u\cdot\na \theta).
$$
Amazingly we have found that $u, \theta$ and $\om$ all satisfy the same damped degenerate wave 
equation only with different nonlinear terms. In comparison with the original system (\ref{bb1}), the new system of wave type equations in (\ref{bb2}) helps unearth all the smoothing and 
stabilization hidden in the original system. The velocity in (\ref{bb1}) involves only vertical dissipation, but the wave structure actually implies that the temperature can stabilize the fluids by creating the horizontal regularization via the coupling and interaction. 

\vskip .1in 
How much regularity and stabilization can the wave structure help create? Our very first effort 
is devoted to understanding this natural question. We focus on the linearized system 
\beq\label{nnn}
\begin{cases}
	\p_{tt} u -( \eta \p_{11} + \nu \p_{22})  \p_t u + \nu \eta \p_{11} \p_{22} u  + \p_{11}\Delta^{-1} u = 0,\\
\p_{tt} \theta -( \eta \p_{11} + \nu \p_{22})  \p_t \theta + \nu \eta \p_{11} \p_{22} \theta + \p_{11}\Delta^{-1} \theta = 0, \\
u(x,0) =u_0(x), \quad \theta(x,0) = \theta_0(x).
\end{cases}
\eeq
To maximally extract the regularity and damping effects from the wave structure, we represent  the solution of (\ref{nnn}) explicitly in terms of kernel functions and the initial data. The two components $u_1$ and $u_2$ of the velocity field have slightly different explicit representations.  

\begin{pro} \label{re}
The solution of (\ref{nnn}) can be explicitly represented as 
\ben
&& u_1(t) =K_1(t) \, u_{10}  + K_2(t)\, \theta_0, \label{tt}\\
&& u_2(t) = K_1(t) \, u_{20} + K_3(t)\, \theta_0, \label{tt0}\\
&& \theta(t) = K_4(t)\, u_{20}  + K_5(t)\,\theta_0, \label{tt1}
\een
where $K_1$ through $K_5$ are Fourier multiplier operators with their symbols given by 
\ben
&& \hskip -.2in K_1(\xi, t) = G_2(\xi, t) - \nu \xi_2^2 G_1(\xi,t), \qquad K_2(\xi,t) = -\frac{\xi_1\xi_2}{|\xi|^2}\, G_1(\xi,t),
\label{kk1}\\
&& \hskip -.2in K_3(\xi,t) = \frac{\xi_1^2}{|\xi|^2} \, G_1(\xi,t), \quad K_4 =- G_1, \quad K_5(\xi,t) = G_2(\xi, t) - \eta \xi_1^2 G_1(\xi,t).\label{kk2}
\een
Here $G_1$ and $G_2$ are two explicit symbols involving the roots $\lambda_1$ and $\lambda_2$ of the characteristic equation
$$
\lambda^2 + (\eta \xi_1^{2}+\nu \xi_2^{2})\lambda+ \nu \eta \xi_1^{2} \xi_2^2+  \frac{\xi_1^2}{|\xi|^2} = 0
$$
or 
$$
\begin{aligned}
& \lambda_1=-\frac{1}{2}(\eta \xi_1^{2}+\nu \xi_2^{2})-\frac{1}{2}\sqrt{(\eta \xi_1^{2}+\nu \xi_2^{2})^2-4\left(\nu \eta \xi_1^{2}\xi_2^2+  \frac{\xi_1^2}{|\xi|^2}\right)},  \\
& \lambda_2=-\frac{1}{2}(\eta \xi_1^{2}+\nu \xi_2^{2})+\frac{1}{2}\sqrt{(\eta \xi_1^{2}+\nu \xi_2^{2})^2-4\left(\nu \eta \xi_1^{2}\xi_2^2+  \frac{\xi_1^2}{|\xi|^2}\right)}.  
\end{aligned}
$$
More precisely, when $\lambda_1\not =\lambda_2$, 
\begin{equation}\label{g1}
G_1(\xi, t) = \frac{e^{\lambda_1 t}-e^{\lambda_2 t}}{\lambda_1- \lambda_2}, \qquad  
G_2(\xi, t) = \frac{\lambda_1 e^{\lambda_2 t}- \lambda_2 e^{\lambda_1t}}{\lambda_1 -\lambda_2}.
\end{equation}	
When $\lambda_1 =\lambda_2$,
\beq\label{g2}
G_1(\xi, t) = t e^{\lambda_1 t}, \qquad G_2(\xi, t) = e^{\lambda_1 t} - \lambda_1 t\, e^{\lambda_1 t}.
\eeq
\end{pro}

\vskip .1in 
In order to understand the regularity and large-time behavior, we need to have precise 
upper bounds on the kernel functions $K_1$ through $K_5$. The behavior of these kernel functions 
depends crucially on the frequency $\xi$ and is nonhomogeneous. In addition, the bounds for 
these kernel functions are anisotropic and are not uniform in different directions. 
The details of these upper bounds and how they are derived are provided in Proposition \ref{432} in Section \ref{sec2}. 

\vskip .1in 
We are able to establish the precise large-time behavior of the solutions to (\ref{nnn}) using the upper bounds for the kernel functions $K_1$ through $K_5$ in Proposition \ref{432}. To reflect the anisotropic behavior of the solutions, we need to employ anisotorpic Sobolev type spaces. For $s\ge 0$ and $\sigma\ge 0$, the anisotropic Sobolev 
space $\mathring H_1^{s, -\sigma}(\mathbb R^2)$ consists of functions $f$ satisfying 
$$
\|f\|_{\mathring H_1^{s, -\sigma}(\mathbb R^2)} = \left(\int_{\mathbb R^2}|\xi|^{2s} \, |\xi_1|^{-2\sigma}  |\widehat{f}(\xi)|^2\,d\xi\right)^{\frac12} <\infty.
$$
Similarly, $\mathring H_2^{s, -\sigma}(\mathbb R^2)$ consists of functions $f$ satisfying 
$$
\|f\|_{\mathring H_2^{s, -\sigma}(\mathbb R^2)} = \left(\int_{\mathbb R^2}|\xi|^{2s} \, |\xi_2|^{-2\sigma}  |\widehat{f}(\xi)|^2\,d\xi\right)^{\frac12} <\infty.
$$
In addition, we write $\mathring H^{s, -\sigma}(\mathbb R^2) = \mathring H_1^{s, -\sigma}(\mathbb R^2)
\cap \mathring H_2^{s, -\sigma}(\mathbb R^2)$ with the norm given by 
$$
\|f\|_{\mathring H^{s, -\sigma}(\mathbb R^2)} =\|f\|_{\mathring H_1^{s, -\sigma}(\mathbb R^2)} + \|f\|_{\mathring H_2^{s, -\sigma}(\mathbb R^2)}.
$$

\begin{thm}\label{main1}
Consider the linearized system in (\ref{nnn}) with the initial data $u_0$ and $\theta_0$ satisfying $\na\cdot u_0 =0$ and 
$$
u_0 \in \mathring H^{0, -\sigma} \cap  \mathring H^{s, -\sigma} \cap \mathring H^{s-2, -\sigma}, \quad\theta_0 \in \mathring H^{0, -\sigma} \cap  \mathring H^{s, -\sigma} \cap \mathring H^{s-1, -\sigma},
$$
where $s\ge 0$ and $\sigma\ge 0$ satisfy $s+ \sigma \ge 2$. Then the corresponding solution $(u, \theta)$ to (\ref{nnn}) satisfies, for some constant $C>0$, 
\beno 
\|u_1(t)\|_{\mathring H^s} &\le& C\, t^{-\frac12 (s+\sigma)}\, \|u_{10}\|_{\mathring H^{0, -\sigma}} + C\, t^{-\frac{\sigma}{2}}\, \|u_{10}\|_{\mathring H^{s, -\sigma}}\\
&& + \,C\, t^{-\frac12 (s+\sigma) +1}\, \|\theta_0\|_{\mathring H^{0, -\sigma}}
+ C\,t^{-\frac12-\frac{\sigma}{2}}\, \|\theta_0\|_{\mathring H^{s-1, -\sigma}}, \\
\|u_2(t)\|_{\mathring H^s} &\le& C\, t^{-\frac12 (s+\sigma)}\, \|u_{20}\|_{\mathring H^{0, -\sigma}} + C\, t^{-\frac{\sigma}{2}}\, \|u_{20}\|_{\mathring H^{s, -\sigma}}\\
&& + \,C\, t^{-\frac12 (s+\sigma) +1}\, \|\theta_0\|_{\mathring H^{0, -\sigma}}
+ C\,t^{-1-\frac{\sigma}{2}}\, \|\theta_0\|_{\mathring H^{s, -\sigma}},\\
\|\theta(t)\|_{\mathring H^s} &\le& \,C\, t^{-\frac12 (s+\sigma) +1}\, \|u_{20}\|_{\mathring H^{0, -\sigma}}
+ C\,t^{-\frac{\sigma}{2}}\, \|u_{20}\|_{\mathring H^{s-2, -\sigma}}\\
&& + \,C\, t^{-\frac12 (s+\sigma)}\, \|\theta_0\|_{\mathring H^{0, -\sigma}} + C\, t^{-\frac{\sigma}{2}}\, \|\theta_0\|_{\mathring H^{s, -\sigma}},
\eeno 
where $\mathring H^s$ denotes the standard homogeneous Sobolev space with its norm defined by 
$$
\|f\|_{\mathring H^s} = \||\xi|^s \, |\widehat f(\xi)|\|_{L^2(\mathbb R^2)}.
$$
\end{thm}

\vskip .1in 
Next we further exploit the effects of stabilizing and regularization of the wave structure 
through the energy method. By forming suitable Lyapunov functional and computing their time evolution, we are able to show that the frequencies away from the two axes in the frequency space decay exponentially 
to zero as $t\to \infty$. To state our result more precisely, we define a frequency cutoff function, for $a_1>0$ and $a_2>0$, 
\beq\label{phi}
\widehat{\varphi}(\xi)= \widehat{\varphi}(\xi_1, \xi_2)=
\begin{cases}
0,  &\quad \text{if} \quad |\xi_1| \leq a_1 \,\,\,\mbox{or}\,\, |\xi_2|\leq a_2, \\
1, &\quad \text{otherwise}.
\end{cases}
\eeq

\vskip .1in 
\begin{thm}\label{main2}
	Let $\nu>0$ and $\eta>0$. Consider the linearized system in (\ref{nnn}) or equivalently
	$$
	\begin{cases}
	\partial_t u_1=  \nu\, \partial_{22}u_1- \Delta ^{-1} \partial_{1}\partial_{2} \theta,  \\
	\partial_t u_2= \nu\, \partial_{22}u_2+ \Delta ^{-1} \partial_{1}\partial_{1} \theta,  \\
	\partial_t \theta = \eta\, \partial_{11} \theta- u_2,\\
	(u_1, u_2, \theta)(x,0) = (u_{01}, u_{02}, \theta_0). 
	\end{cases}
    $$
	 Let $(u, \theta)$ be the corresponding solution.  The Fourier frequency piece of $(u, \theta)$ away from the two axes of the frequency space decays exponentially in time to zero. 
	More precisely, if $(u_0, \theta_0) \in H^2(\mathbb R^2)$ with $\na\cdot u_0=0$, then there is constant $C_0 =C_0(\nu, \eta, a_1, a_2)$ such that, for all $t\ge 0$, 
	\ben
	&& \|\partial_{t}(\varphi * u)(t)\|_{L^{2}}^{2} + \|(\varphi * u)(t)\|_{H^1}^{2}
    \le \,C\, (\|\varphi * u_0\|_{H^{2}}^{2} + \|\varphi *\theta_0\|_{L^2}^2)\, e^{-C_0 t},\quad 
    \label{ss2}\\
	&& \|\partial_{t}(\varphi * \theta)(t)\|_{L^{2}}^{2} + \|(\varphi * \theta)(t)\|_{H^1}^{2}
	\le \,C\, (\|\varphi * \theta_0\|_{H^{2}}^{2} + \|\varphi *u_0\|_{L^2}^2)\, e^{-C_0 t},\quad 
	\label{ss3}
	\een
	where $\varphi$ is as defined in (\ref{phi}) and $C=C(\nu, \eta, a_1, a_2)>0$ is a constant. 
\end{thm}

\vskip .1in 
We now turn our attention to the main result of this paper, the nonlinear stability on (\ref{bb1}). As we have explained before, the major obstacle is  how to obtain 
a suitable upper bound on the nonlinear term from the momentum equation, namely (\ref{hard}). 
This is the main reason why the stability problem on the 2D Navier-Stokes equations with 
only one-directional dissipation remains open. However, for the coupled nonlinear system in 
(\ref{bb1}), the smoothing and stabilizing effect of the temperature on the fluid velocity makes
the nonlinear stability possible. In fact, we are able to prove the following theorem. 
\begin{thm} \label{main3}
Consider (\ref{bb1}) with $\nu>0$ and $\eta>0$. Assume the initial data $(u_0, \theta_0)$ is in $H^2(\mathbb R^2)$ with $\na\cdot u_0=0$. Then there exists 
$\varepsilon =\varepsilon(\nu, \eta) >0$ such that, if $(u_0, \theta_0)$ satisfies
$$
\| u_0\|_{H^2}+\| \theta_0\|_{H^2}\leq \varepsilon,
$$
then (\ref{bb1}) has a unique global solution $(u, \theta)$ satisfying, for any $t>0$, 
\beno
&& \|u(t)\|_{H^2}^{2}+ \|\theta(t)\|_{H^2}^{2} 
+ \nu\, \int_{0}^{t} \|\partial_2 u\|_{H^2}^{2}\,d\tau \\
&& \qquad   + \eta \, \int_{0}^t \|\partial_1 \theta\|_{H^2}^{2} \, d\tau + C(\nu, \eta)\,   \int_{0}^t \|\partial_1 u_2\|_{L^2}^{2}\, d\tau \leq \, C\; \varepsilon^2,
\eeno
where $C(\nu, \eta)>0$ and $C>0$ are constants.
\end{thm} 

In order to prove Theorem \ref{main3}, we need to exploit the extra regularization due to the 
wave structure in (\ref{bb2}). In particular, the 
control on the time integral of the horizontal derivative of the velocity field, namely
\beq\label{cru}
\int_0^t \|\p_1 u(\tau)\|_{L^2}^2\,d\tau 
\eeq
plays a crucial role in the proof. Clearly the uniform boundedness of (\ref{cru}) is 
not a consequence of the vertical dissipation in the velocity equation but due to the interaction
with the temperature equation. We use the bootstrapping argument to prove the boundedness 
of (\ref{cru}) and the stability of the solution simultaneously. A general statement on the bootstrapping principle can be found in \cite[p.21]{Tao2006}. To achieve this goal, we first 
construct a suitable energy functional 
\ben
E(t) &=& \max_{0\leq \tau \leq t} (\|u(\tau)\|_{H^2}^{2} +\|\theta(\tau)\|_{H^2}^{2})  + 2 \nu \int_{0}^{t} \|\partial_2 u\|_{H^2}^{2}d\tau \notag\\
&& + 2\eta \int_{0}^t \|\partial_1 \theta\|_{H^2}^{2} d\tau + \delta \int_{0}^{t}\|\partial_1 u_2\|_{L^2}^{2}\, d\tau, \label{ee}
\een
where $\delta>0$ is a suitably selected parameter. We then show that $E(t)$ satisfies 
\beq\label{ine}
E(t) \le C\, E(0) + C\, E(t)^{\frac32}.
\eeq
Our main efforts are devoted to proving (\ref{ine}). In particular, we need to estimate the difficult term (\ref{hard}). A suitable upper bound can now be achieved due to the inclusion 
of (\ref{cru}) in the energy function. $\delta>0$ is chosen to be sufficiently small so that 
some of the terms generated in the estimating of (\ref{cru}) can be majorized by the dissipative terms. We leave more technical details on how to bound (\ref{hard}) and other terms to Section \ref{sec4}. In order to take advantage of the anisotropic 
dissipation, the estimates are performed via anisotropic tools including an anisotropic triple product upper bound as stated in the following lemma taken from \cite{caowu2011}. 
\begin{lem}\label{tri}
	Assume that $f$, $g$, $\p_2 g$, $h$ and $\p_1h$ are all in $L^2(\mathbb{R}^2)$. Then,
	for some constant $C>0$,  
	$$ \int_{\mathbb R^2} |fgh|\,dx \leq C \|f\|_{L^2}\|g\|_{L^2}^{\frac{1}{2}} \|\p_2g\|_{L^2}^{\frac{1}{2}} \|h\|_{L^2}^{\frac{1}{2} }\|\p_1h\|_{L^2}^{\frac{1}{2}}.$$
\end{lem}
Once (\ref{ine}) is established, the bootstrapping argument then implies that, if $E(0)$ is sufficiently small or equivalently 
$$
\|(u_0, \theta_0)\|_{H^2} \le \varepsilon
$$
for some sufficiently small $\varepsilon>0$, then $E(t)$ remains uniformly small for all time, 
namely 
$$
E(t)  \le C\, \varepsilon^2
$$
for a constant $C>0$ and for all $t\ge 0$. Details on the application of the bootstrapping argument will be provided in the proof of Theorem \ref{main3} in Section \ref{sec4}. 

\vskip .1in 
Finally we remark that, due to its importance in 
geophysics and astrophysics, the stability problem on the hydrostatic balance has recently attracted considerable interests. When 
the Boussinesq system does not involve full kinematic dissipation and thermal diffusion, 
the stability problem can be extremely difficult. Several recent work has made 
progress. Doering, Wu, Zhao and Zheng \cite{DWZZ} solved the stability problem 
on the 2D Boussinesq 
system with full velocity dissipation but without thermal diffusion in a bounded domain 
with stress-free boundary condition. A follow-up work by Tao, Wu, Zhao and Zheng \cite{TWZZ} 
was able
to establish the precise large-time behavior of the stable solutions obtained in \cite{DWZZ}. Castro, Cordoba and Lear \cite{CCL} investigated the stability problem of the 2D Boussinesq system when 
the velocity involves a damping term and obtained the asymptotic stability for a trip domain.
We also mention that the study on the stability problem on the Boussinesq equations near the shear flow, another physically important steady state, has also gained momentum (see \cite{DWZ, TW, Zil}). 

\vskip .1in 
The rest of this paper is naturally divided into three sections. Section \ref{sec2} provides 
the proofs of  of Proposition \ref{re} and Theorem \ref{main1}. Section \ref{sec3} proves Theorem \ref{main2} while Section \ref{sec4} presents the proof of Theorem \ref{main3}.

\vskip .3in
\section{Proofs of Proposition \ref{re} and Theorem \ref{main1}}
\label{sec2}

This section is devoted to the proofs of Proposition \ref{re} and Theorem \ref{main1}.
Proposition \ref{re} represents the solution to the linearized system in (\ref{nnn}) in 
terms of the initial data and several kernel functions. Its proof relies
on a lemma that solves the degenerate damped wave equation explicitly. 
The decay estimates in Theorem \ref{main1} are based on the upper bounds for 
the kernel functions in the representation of solutions obtained in Proposition \ref{re}. 
The upper bounds are derived in Proposition \ref{432} prior to the proof of Theorem \ref{main1}. 

\begin{lem}\label{55}
	Assume that $f$ satisfies the damped degenerate wave type equation 
     \beq\label{ww}
     \begin{cases}
	\partial_{tt}f-(\nu \partial_{22}+\eta \partial_{11}) \partial_t f+ \eta \nu \partial_{11}\partial_{22}f + \p_{11}\Delta^{-1} f=F, \\
	f(x,0)= f_0(x), \quad (\partial_t f)(x,0) = f_1(x).
	\end{cases}
	\eeq
	Then f can be explicitly represented as
	\begin{equation}\label{wavetypef}
	f(t) = G_1(t)\, f_1 + G_2(t) \,f_0 + \int_0^t G_1(t-\tau)\,F(\tau)\,d\tau,
	\end{equation}
	where $G_1$ and $G_2$ are two Fourier multiplier operators with their symbols given by 
	\begin{equation}\label{la}
	G_1(\xi, t) = \frac{e^{\lambda_1 t}-e^{\lambda_2 t}}{\lambda_1- \lambda_2}, \qquad  
	G_2(\xi, t) = \frac{\lambda_1 e^{\lambda_2 t}- \lambda_2 e^{\lambda_1t}}{\lambda_1 -\lambda_2}
	\end{equation}	
	with $\lambda_1$ and $\lambda_2$ being the roots of the characteristic equation
	\beq\label{ch}
\lambda^2 + (\eta \xi_1^{2}+\nu \xi_2^{2})\lambda+ \nu \eta \xi_1^{2} \xi_2^2+  \frac{\xi_1^2}{|\xi|^2} = 0
	\eeq
	or 
	\beq\label{ll}
	\begin{aligned}
	& \lambda_1=-\frac{1}{2}(\eta \xi_1^{2}+\nu \xi_2^{2})-\frac{1}{2}\sqrt{(\eta \xi_1^{2}+\nu \xi_2^{2})^2-4\left(\nu \eta \xi_1^{2}\xi_2^2+  \frac{\xi_1^2}{|\xi|^2}\right)},  \\
	& \lambda_2=-\frac{1}{2}(\eta \xi_1^{2}+\nu \xi_2^{2})+\frac{1}{2}\sqrt{(\eta \xi_1^{2}+\nu \xi_2^{2})^2-4\left(\nu \eta \xi_1^{2}\xi_2^2+  \frac{\xi_1^2}{|\xi|^2}\right)}.  
	\end{aligned}
	\eeq
	When $\lambda_1= \lambda_2$, (\ref{wavetypef}) remains valid if we replace  $G_1$ and $G_2$ in (\ref{la}) by their corresponding limit form, namely, 
	$$
	G_1({\xi, t})=\lim_{\lambda_2 \to \lambda_1} \frac{e^{\lambda_1 t}-e^{\lambda_2 t}}{\lambda_1- \lambda_2}= te^{\lambda_1 t}
	$$
	and 
	$$
	G_2({\xi, t})=\lim_{\lambda_2 \to \lambda_1}\frac{\lambda_1 e^{\lambda_2 t}- \lambda_2 e^{\lambda_1t}}{\lambda_1 -\lambda_2} = e^{\lambda_1t}- \lambda_1t e^{\lambda_1t}. 
	$$
\end{lem}

\vskip .1in 
\begin{proof}[Proof of Lemma \ref{55}] We first focus on the case when $F\equiv 0$. Since $\lambda_1(\xi)$ and $\lambda_2(\xi)$ are the roots of the characteristic equation in (\ref{ch}), 
we can decompose the second-order differential operator as follows, 
\begin{equation}\label{dd1}
(\partial_t - \lambda_1(D)) (\partial_t - \lambda_2(D)) f= 0 
\end{equation}
and 
\begin{equation}\label{dd2}
(\partial_t - \lambda_2(D)) (\partial_t - \lambda_1(D)) f= 0,
\end{equation}
where $\lambda_1(D)$ and $\lambda_1(D)$ are the Fourier multiplier operators with their symbols given by $\lambda_1(\xi)$ and $\lambda_2(\xi)$, or 
$$
\begin{aligned}
& \lambda_1(D)=\frac{1}{2}(\nu \partial_{22}+ \eta \partial_{11})-\frac{1}{2}\sqrt{(\nu \partial_{22}+ \eta \partial_{11})^2-4(\nu \eta \partial_{1122}+  \partial_{11}\Delta^{-1})},  \\
& \lambda_2(D)=\frac{1}{2}(\nu \partial_{22}+ \eta \partial_{11})+\frac{1}{2}\sqrt{(\nu \partial_{22}+ \eta \partial_{11})^2-4(\nu \eta \partial_{1122}+ \partial_{11}\Delta^{-1})}.  
\end{aligned}
$$
We can rewrite (\ref{dd1}) and (\ref{dd2}) into two systems 
\begin{equation}\label{firstorderfg}
\begin{cases}
(\partial_t - \lambda_1(D)) g= 0, \\
(\partial_t - \lambda_2(D)) f= g 
\end{cases}
\end{equation}
and 
\begin{equation}\label{firstorderfh}
\begin{cases}
(\partial_t - \lambda_2(D)) h= 0, \\
(\partial_t - \lambda_1(D)) f= h.
\end{cases}
\end{equation}
By taking the difference of the second equations of (\ref{firstorderfg}) and (\ref{firstorderfh}), we obtain
$$
(\lambda_1(D)- \lambda_2(D)) f= g-h
$$
or 
\begin{equation}\label{differencegh}
f= ((\lambda_1(D)- \lambda_2(D)))^{-1} (g-h).
\end{equation}
Solving the first equations of (\ref{firstorderfg}) and (\ref{firstorderfh}) yields,
\begin{equation}\label{initialg}
g(t)= g(0)\, e^{\lambda_1(D)\, t} =  ( (\partial_t f)(0)- \lambda_2(D)f(0)) \, e^{\lambda_1(D)t}
\end{equation}
and 
\begin{equation}\label{initialh}
h(t)= h(0)\, e^{\lambda_2(D)\, t} =  ( (\partial_t f)(0)- \lambda_1(D)f(0)) \, e^{\lambda_2(D)t},
\end{equation}
where we have used second equations of (\ref{firstorderfg}) and (\ref{firstorderfh}) to obtain the initial data $g(0)$ and $h(0)$. Inserting (\ref{initialg}) and (\ref{initialh}) in (\ref{differencegh}) leads to
$$
\begin{aligned}
f(t) &= (\lambda_1(D)- \lambda_2(D))^{-1}\;\Big((e^{\lambda_1(D)t}- e^{\lambda_2(D)t})\;(\partial_t f)(0) \\
& \quad +(\lambda_1(D) e^{\lambda_2(D)t}-\lambda_2(D) e^{\lambda_1(D)t})\,  f(0)\Big) \\
&= G_1\,f_1+  G_2\,f_0,
\end{aligned}
$$
where 
$$
G_1 = \frac{e^{\lambda_1(D)t}- e^{\lambda_2(D)t}}{\lambda_1(D)-\lambda_2(D)},\qquad 
G_2 =  \frac{\lambda_1(D)e^{\lambda_2(D)t}-\lambda_2(D) e^{\lambda_2(D)t}}{\lambda_1(D)-\lambda_2(D)}.
$$
When $F$ in (\ref{ww}) is not identically zero, the formula in (\ref{wavetypef}) is obtained by Duhamel's principle. This completes the proof of Lemma \ref{55}.
\end{proof}

\vskip .1in 
We are now ready to prove Proposition \ref{re}.
\begin{proof}[Proof of Proposition \ref{re}] This is a direct consequence of Lemma \ref{55}. In fact, according to Lemma \ref{55}, 
\beq\label{xz}
u(t) = G_2(t)\, u_0 + G_1(t)\, (\p_t u)(x,0), \quad 
\theta(t) = G_2(t)\, \theta_0 +  G_1(t)\, (\p_t \theta)(x,0).
\eeq
Since $u$ and $\theta$ satisfy the original linearized equations, 
\beno
&& \p_t u_1 = \nu\, \p_{22} u_1 -  \p_1\p_2 \Delta^{-1} \theta, \\
&& \p_t u_2 = \nu\, \p_{22} u_2 +   \p_{11} \Delta^{-1} \theta,\\
&&\p_t \theta = \eta\, \p_{11} \theta - u_2,
\eeno
we obtain 
\beno
&& (\p_t u_1)(x,0) = \nu\, \p_{22} u_{10}-  \p_1 \p_2 \Delta^{-1} \theta_0,\\
&& (\p_t u_2(x,0) = \nu\, \p_{22} u_{20} +   \p_{11} \Delta^{-1} \theta_0,\\
&& (\p_t \theta)(x,0) = \eta\,\p_{11} \theta_0 - u_{20}.
\eeno
Inserting them  in (\ref{xz}), we obtain 
\beno
&& u_1(t) = (G_2(t)\,+ \nu \p_{22} G_1)\, u_{10} - \p_1\p_2 \Delta^{-1}\,G_1\, \theta_0, \\
&& u_2(t) = (G_2(t)\,+ \nu \p_{22} G_1)\, u_{20} +  \p_{11} \Delta^{-1} \, G_1\, \theta_0, \\
&& \theta(t) = -G_1\, u_{20} + (G_2 + \eta \p_{11} G_1) \theta_0,
\eeno
which are the representations in (\ref{tt}), (\ref{tt0}) and (\ref{tt1}). This completes the proof of Proposition \ref{re}. 
\end{proof}

\vskip .1in 
In order to prove Theorem \ref{main1}, we need to understand the behavior 
of the kernel functions $K_1$ through $K_5$. Clearly their behavior depends on the frequency $\xi$. In order to obtain a definite behavior for each kernel function, we need to divide
the whole frequency space $\mathbb R^2$ into subdomains. The following proposition specifies 
these subdomains and the behavior of the kernel functions. 
\begin{pro}\label{432}
	Assume the kernel functions $K_1$ through $K_5$ are given by (\ref{kk1}) and (\ref{kk2}) with $G_1$ and $G_2$ defined in (\ref{g1}) and (\ref{g2}). Set 
	\beno
	&& S_1 = \left\{\xi =(\xi_1, \xi_2)\in \mathbb R^2, \, \nu \eta \xi_1^2\,\xi_2^2 + \xi_1^2\,|\xi|^{-2} \ge \frac3 {16} (\nu \xi_2^2 + \eta \xi_1^2)^2 \right\},\\
	&& S_2 = \mathbb R^2 \setminus S_1.
	\eeno 
	The kernel functions $K_1$ through $K_5$ can then be bounded as follows. 
	\begin{enumerate} 
	\item[(a)] Let $\xi\in S_1$. Then 
	$$
	Re \lambda_1 \le -\frac12(\nu \xi_2^2 + \eta \xi_1^2), \qquad 	Re  \lambda_2 \le -\frac14(\nu \xi_2^2 + \eta \xi_1^2),
	$$
	where $Re$ denotes the real part, and, for constants $c_0>0$ and $C>0$, 
	\ben
	&& |K_1(\xi,t)|, |K_5(\xi,t)|\le \,C\, e^{-c_0 |\xi|^2 t}, \label{gb1} \\
	&& |K_2(\xi,t)| , |K_3(\xi,t)|,  |K_4(\xi,t)|\le \,C\, t\, e^{-c_0 |\xi|^2 t}. \label{gb2}
	\een
	
	\vskip .1in 
	\item[(b)] Let $\xi\in S_2$. Then 
	$$
	\lambda_1 \le -\frac34(\nu \xi_2^2 + \eta \xi_1^2), \qquad \lambda_2 \le -\frac{\nu \eta \xi_1^2 \xi_2^2 + \xi_1^2 |\xi|^{-2}}{\nu \xi_2^2 + \eta \xi_1^2},
	$$
	\beq\label{gb3}
	|K_1|, \, |K_5| \le C\, e^{-\frac34(\nu \xi_2^2 + \eta \xi_1^2) t} + C\, e^{-\frac{\nu \eta \xi_1^2 \xi_2^2 + |\xi_1|^2|\xi|^{-2}}{\nu \xi_2^2 + \eta \xi_1^2} t}
	\eeq
	and 
	 \ben
  &&	|K_2| \le \frac{C|\xi_1||\xi_2|}{|\xi|^4} \,e^{-c_0 \,|\xi|^2 t} + \frac{C|\xi_1||\xi_2|}{|\xi|^4} \, e^{-c_0 \frac{\xi_1^2 \xi_2^2}{|\xi|^2} t}\, e^{-c_0\,\frac{\xi_1^2}{|\xi|^4} t}, \label{gb4}\\
  &&	|K_3| \le \frac{C|\xi_1|^2}{|\xi|^4} \,e^{-c_0 \,|\xi|^2 t} + \frac{C|\xi_1|^2}{|\xi|^4} \, e^{-c_0 \frac{\xi_1^2 \xi_2^2}{|\xi|^2} t}\, e^{-c_0\,\frac{\xi_1^2}{|\xi|^4} t},\notag\\
&& |K_4| \le \frac{C}{|\xi|^2} \,e^{-c_0 \,|\xi|^2 t} + \frac{C}{|\xi|^2} \, e^{-c_0 \frac{\xi_1^2 \xi_2^2}{|\xi|^2} t}\, e^{-c_0\,\frac{\xi_1^2}{|\xi|^4} t}.  \notag
	\een
	\end{enumerate}
\end{pro}

\begin{proof}
	To prove the bounds in (a), we further divide $S_1$ into two subsets,
	\beno
	&& S_{11} = \left\{\xi \in S_1, (\nu \xi_2^2 + \eta \xi_1^2)^2 \ge 4 (\nu \eta \xi_1^2 \xi_2^2 + |\xi_1|^2|\xi|^{-2}) \right\}, \\
	&& S_{12} = S_1\setminus S_{11}.
	\eeno
	For any $\xi \in S_{11}$, 
	$$
	0\le (\nu \xi_2^2 + \eta \xi_1^2)^2 -4 (\nu \eta \xi_1^2 \xi_2^2 + |\xi_1|^2|\xi|^{-2}) 
	\le \frac14 (\nu \xi_2^2 + \eta \xi_1^2)^2.
	$$
	According to the formula for $\lambda_1$ and $\lambda_2$ in (\ref{ll}), $\lambda_1$ and $\lambda_2$ are real and satisfy
	$$
	\lambda_1 \le - \frac12 (\nu  \xi_2^2 + \eta \xi_1^2), \qquad \lambda_2 \le -\frac14 (\nu  \xi_2^2 + \eta \xi_1^2).
	$$
	By the mean-value theorem, for a constant $C>0$, 
	\beq\label{g1b}
	|G_1| = \left|\frac{e^{\lambda_1 t} -e^{\lambda_2 t}}{\lambda_1 -\lambda_2} \right|
	\le t \, e^{- C\, |\xi|^2 t}.
	\eeq
    Writing $G_2$ in (\ref{g1}) as 
    $$
    G_2 =e^{\lambda_1 t} - \lambda_1 G_1
    $$
    and using the simple fact that $x^m \, e^{-x} \le C(m)$ for any $x\ge 0$ and $m\ge 0$, we can 
    bound $K_1$ and $K_5$ as follows, 
    \beno
    |K_1| &\le& |G_2| + \nu |\xi_2^2|\, |G_1| \le e^{-c_0 |\xi|^2 t} + C\, |\xi|^2 \,t \, e^{-C\, |\xi|^2 t} + \nu |\xi_2^2|\, t \, e^{-C\, |\xi|^2 t} \\
    &\le& C\, e^{-c_0 |\xi|^2 t},\\
    |K_5| &\le&  |G_2| + \eta |\xi_1^2|\, |G_1| \le C\, e^{-c_0 |\xi|^2 t},
    \eeno
    where $C>0$ and $c_0>0$ are constants. The bounds $K_2$,  $K_3$ and $K_4$ 
    follow directly from (\ref{g1b}). For $\xi\in S_{12}$, 
    $$
    (\nu \xi_2^2 + \eta \xi_1^2)^2 < 4 (\nu \eta \xi_1^2 \xi_2^2 + |\xi_1|^2|\xi|^{-2})
    $$
    and, as a consequence, $\lambda_1$ and $\lambda_2$ are complex numbers,
    \beno
    && \lambda_1 = -\frac12(\nu \xi_2^2 + \eta \xi_1^2) - \frac{i}{2}\, \sqrt{4 (\nu \eta \xi_1^2 \xi_2^2 + |\xi_1|^2|\xi|^{-2}) - (\nu \xi_2^2 + \eta \xi_1^2)^2},\\
    &&  \lambda_2 = -\frac12(\nu \xi_2^2 + \eta \xi_1^2) + \frac{i}{2}\, \sqrt{4 (\nu \eta \xi_1^2 \xi_2^2 + |\xi_1|^2|\xi|^{-2}) - (\nu \xi_2^2 + \eta \xi_1^2)^2}.
    \eeno 
    Then 
    $$
    Re\, \lambda_1 = Re\, \lambda_2 = -\frac12(\nu \xi_2^2 + \eta \xi_1^2).
    $$
    In addition, 
    \beno
    |G_1| &=& \left|\frac{e^{\lambda_1 t} -e^{\lambda_2 t}}{\lambda_1 -\lambda_2} \right| \\
    &=& e^{-\frac12(\nu \xi_2^2 + \eta \xi_1^2) t}\, \left|\frac{\sin\left(t \sqrt{4 (\nu \eta \xi_1^2 \xi_2^2 + |\xi_1|^2|\xi|^{-2}) - (\nu \xi_2^2 + \eta \xi_1^2)^2}\right)}{\sqrt{4 (\nu \eta \xi_1^2 \xi_2^2 + |\xi_1|^2|\xi|^{-2}) - (\nu \xi_2^2 + \eta \xi_1^2)^2}}\right| \\
    &\le& t\, e^{-\frac12(\nu \xi_2^2 + \eta \xi_1^2) t}.
    \eeno
    The desired upper bounds for $K_1$ through $K_5$ then follow as before. 
    
    \vskip .1in 
    We now prove the bounds in (b). For $\xi \in S_2$, 
    \beq\label{0b}
    (\nu \xi_2^2 + \eta \xi_1^2)^2 -4 (\nu \eta \xi_1^2 \xi_2^2 + |\xi_1|^2|\xi|^{-2}) \ge \frac14 (\nu \xi_2^2 + \eta \xi_1^2)^2.
    \eeq
    Then $\lambda_1$ and $\lambda_2$ are both real. Clearly,  $\lambda_1$ satisfies
    \beq\label{11b}
    \lambda_1 \le -\frac34(\nu \xi_2^2 + \eta \xi_1^2).
    \eeq
    To obtain the upper bound for $\lambda_2$, we try to make the terms in the representation of $\lambda_2$ have the same sign and obtain 
    \ben
    \lambda_2 &=& -\frac12\left((\nu \xi_2^2 + \eta \xi_1^2) -\sqrt{(\nu \xi_2^2 + \eta \xi_1^2)^2 -4(\nu \eta \xi_1^2 \xi_2^2 + |\xi_1|^2|\xi|^{-2})}\right) \notag\\
    &=& -2 \frac{\nu \eta \xi_1^2 \xi_2^2 + |\xi_1|^2|\xi|^{-2}}{\nu \xi_2^2 + \eta \xi_1^2
    	+\sqrt{(\nu \xi_2^2 + \eta \xi_1^2)^2 -4(\nu \eta \xi_1^2 \xi_2^2 
    		+ |\xi_1|^2|\xi|^{-2})}} \notag\\
    &\le& -\frac{\nu \eta \xi_1^2 \xi_2^2 + |\xi_1|^2|\xi|^{-2}}{\nu \xi_2^2 + \eta \xi_1^2}.
    \label{2b}
    \een
    It then follows from (\ref{0b}), (\ref{11b}) and (\ref{2b}) that
    \beno
    |G_1| &\le& \frac1{\sqrt{(\nu \xi_2^2 + \eta \xi_1^2)^2 -4(\nu \eta \xi_1^2 \xi_2^2 
    	+ |\xi_1|^2|\xi|^{-2})}} \\
    && \times\left(e^{-\frac34(\nu \xi_2^2 + \eta \xi_1^2) t} + e^{-\frac{\nu \eta \xi_1^2 \xi_2^2 + |\xi_1|^2|\xi|^{-2}}{\nu \xi_2^2 + \eta \xi_1^2} t}\right)\\
    &\le& \frac2{\nu \xi_2^2 + \eta \xi_1^2} \left(e^{-\frac34(\nu \xi_2^2 + \eta \xi_1^2) t} + e^{-\frac{\nu \eta \xi_1^2 \xi_2^2 + |\xi_1|^2|\xi|^{-2}}{\nu \xi_2^2 + \eta \xi_1^2} t}\right)\\
    &\le& \frac{C}{|\xi|^2} \,e^{-c_0 \,|\xi|^2 t} + \frac{C}{|\xi|^2} \, e^{-c_0 \frac{\xi_1^2 \xi_2^2}{|\xi|^2} t}\, e^{-c_0\,\frac{\xi_1^2}{|\xi|^4} t}\\
    &:=& M(\xi,t),
    \eeno
    where $c_0>0$ is a constant. Therefore, 
    $$
    |K_2| \le \frac{C|\xi_1||\xi_2|}{|\xi|^4} \,e^{-c_0 \,|\xi|^2 t} + \frac{C|\xi_1||\xi_2|}{|\xi|^4} \, e^{-c_0 \frac{\xi_1^2 \xi_2^2}{|\xi|^2} t}\, e^{-c_0\,\frac{\xi_1^2}{|\xi|^4} t}, 
    $$
    $$
    |K_3| \le \frac{C|\xi_1|^2}{|\xi|^4} \,e^{-c_0 \,|\xi|^2 t} + \frac{C|\xi_1|^2}{|\xi|^4} \, e^{-c_0 \frac{\xi_1^2 \xi_2^2}{|\xi|^2} t}\, e^{-c_0\,\frac{\xi_1^2}{|\xi|^4} t}
    $$
    and 
    $$
    |K_4| \le \frac{C}{|\xi|^2} \,e^{-c_0 \,|\xi|^2 t} + \frac{C}{|\xi|^2} \, e^{-c_0 \frac{\xi_1^2 \xi_2^2}{|\xi|^2} t}\, e^{-c_0\,\frac{\xi_1^2}{|\xi|^4} t}. 
    $$
    $K_1$ is bounded by 
    \beno
    |K_1| &\le&  |G_2| + \nu |\xi_2^2|\, |G_1| \le e^{\lambda_1 t} \le e^{\lambda_1 t} + |\lambda_1| |G_1| + \nu |\xi_2^2|\, |G_1|\\ 
    &\le& C\, e^{-\frac34(\nu \xi_2^2 + \eta \xi_1^2) t} + C\, e^{-\frac{\nu \eta \xi_1^2 \xi_2^2 + |\xi_1|^2|\xi|^{-2}}{\nu \xi_2^2 + \eta \xi_1^2} t}.
    \eeno
    $K_5$ shares the same bound.
    This completes the proof of Proposition \ref{432}. 
\end{proof}

\vskip .1in 
In order to prove Theorem \ref{main1}, we recall a lemma that provides an explicit  decay rate for the heat kernel associated with a fractional Laplacian $\Lambda^\alpha\ (\alpha\in \mathbb{R}) $. Here the fractional Laplacian operator can be defined through the Fourier transform
\beq\label{fr}
\widehat{\Lambda^\alpha f} (\xi) =|\xi|^{\alpha}\widehat{f}(\xi).
\eeq
The proof of the Lemma can be found in  many references (see, e.g., \cite{Wu01}).

\begin{lem} \label{fb}
	Let $\alpha \ge 0$, $\beta>0$ and {$1\le q \le p \le \infty$}. Then there exists a constant $C$ such that, for any $t>0$,
	$$
	\|\Lambda^\alpha e^{-\Lambda^\beta t}f\|_{L^p(\mathbb R^d)}\le \,C\, t^{-\frac{\alpha}{\beta} - \frac{d}{\beta}(\frac1q-\frac1p)}\, \|f\|_{L^q(\mathbb R^d)}.
	$$
\end{lem}

In addition to the fractional operator defined in (\ref{fr}), we also use the fractional operators $\Lambda_i^\sigma$ with $i=1,2$ defined by 
$$
\widehat{\Lambda_i^\sigma f}(\xi) = |\xi_i|^\sigma \widehat f(\xi), \quad \xi=(\xi_1, \xi_2).
$$
We are now ready to prove Theorem \ref{main1}. 

\begin{proof}[Proof of Theorem \ref{main1}] Taking the $\mathring H^s$-norm of $u_1$ in (\ref{tt}), applying Plancherel's theorem and dividing the spatial domain $\mathbb R^2$ as in Proposition \ref{re}, we obtain 
\beno 
\|u_1(t)\|_{\mathring H^s(\mathbb R^2)} &\le& \|\Lambda^s K_1(t) u_0\|_{L^2(\mathbb R^2)} + \|\Lambda^s K_2(t) \theta_0\|_{L^2(\mathbb R^2)} \\
&\le& C\, \||\xi|^s\,K_1(\xi,t) \widehat{u_0}(\xi)\|_{L^2(S_1)} + C\,\||\xi|^s\,K_1(\xi,t) \widehat{u_0}(\xi)\|_{L^2(S_{2})}\\
&& + C\, \||\xi|^s\,K_2(\xi,t) \widehat{\theta_0}(\xi)\|_{L^2(S_1)} + C\,\||\xi|^s\,K_2(\xi,t) \widehat{\theta_0}(\xi)\|_{L^2(S_{2})}.
\eeno
To bound the terms on the right-hand side, we invoke the upper 
bounds for $K_1$ and $K_2$ obtained in Proposition \ref{432}. By (\ref{gb1}) in Proposition \ref{432}, Plancherel's theorem and Lemma \ref{fb},
\ben
\||\xi|^s\,K_1(\xi,t) \widehat{u_0}(\xi)\|_{L^2(S_1)} &\le& C\, \||\xi|^s\, e^{-c_0 |\xi|^2 t}\,\widehat{u_0}(\xi)\|_{L^2(S_1)} \notag \\
&=& C\, \||\xi|^s\,|\xi_1|^{\sigma} e^{-c_0 |\xi|^2 t}\,\,|\xi_1|^{-\sigma} \widehat{u_0}(\xi)\|_{L^2(S_1)}  \notag\\
&\le&  C\, \||\xi|^{s+\sigma} e^{-c_0 |\xi|^2 t}\,\,|\xi_1|^{-\sigma} \widehat{u_0}(\xi)\|_{L^2({S_1})} \notag \\
&=&  C\, \| \Lambda^{s+\sigma} e^{c_0 \Delta  t}\,\,\Lambda_1^{-\sigma} u_0\|_{L^2(\mathbb R^2)}
\notag \\
&\le& C\, t^{-\frac12 (s+\sigma)}\, \|\Lambda_1^{-\sigma}u_0\|_{L^2(\mathbb R^2)}. \label{ty}
\een
By (\ref{gb3}) in Proposition \ref{432}, 
\beno
\||\xi|^s\,K_1(\xi,t) \widehat{u_0}(\xi)\|_{L^2(S_{2})} &\le& C\, \||\xi|^s\,e^{-c_0 |\xi|^2 t}\,\widehat{u_0}(\xi)\|_{L^2(S_2)}\\
&& \quad  
+ C\,\||\xi|^s\, e^{-\frac{\nu \eta \xi_1^2 \xi_2^2+ |\xi_1|^2|\xi|^{-2}}{\nu \xi_2^2 + \eta \xi_1^2} t}\widehat{u_0}(\xi) \|_{L^2(S_{2})}.
\eeno 
The first part can be bounded the same way as (\ref{ty}). To give a precise upper bound on
the second part, we divide the consideration into two cases: $\xi\in S_{21}$ and $\xi\in S_{22}$, 
where 
$$
S_{21} = \left\{\xi \in S_2,\, |\xi_1| \ge |\xi_2 | \right\}, \quad S_{22} =\left\{\xi \in S_2,\, |\xi_1| < |\xi_2|\right\}
$$
with $S_2$ being defined as in Proposition \ref{432}. For $\xi\in S_{21}$, 
\beq\label{hh11}
-\frac{\nu \eta \xi_1^2 \xi_2^2+ |\xi_1|^2|\xi|^{-2}}{\nu \xi_2^2 + \eta \xi_1^2}  \le -C\, |\xi_2|^2 - C\, |\xi_1|^2|\xi|^{-4} \le -C\, |\xi_2|^2
\eeq
and for $\xi\in S_{22}$, 
\beq\label{hh2}
-\frac{\nu \eta \xi_1^2 \xi_2^2+ |\xi_1|^2|\xi|^{-2}}{\nu \xi_2^2 + \eta \xi_1^2}  \le -C\, |\xi_1|^2 - C\, |\xi_1|^2|\xi|^{-4} \le -C\, |\xi_1|^2.
\eeq
Therefore, 
\beno 
&& \||\xi|^s\, e^{-\frac{\nu \eta \xi_1^2 \xi_2^2+ |\xi_1|^2|\xi|^{-2}}{\nu \xi_2^2 + \eta \xi_1^2} t}\widehat{u_0}(\xi) \|_{L^2(S_{2})}\\ &\le& C\, \||\xi|^s\, e^{-C\, |\xi_2|^2 t}\, \widehat{u_0}(\xi) \|_{L^2(S_{21})}  + C\,  \||\xi|^s\, e^{-C\, |\xi_1|^2 t}\, \widehat{u_0}(\xi) \|_{L^2(S_{22})}\\
 &\le& C\, \| |\xi|^s\,|\xi_2|^{\sigma} e^{-C\, |\xi_2|^2 t}\,|\xi_2|^{-\sigma}\, \widehat{u_0}(\xi) \|_{L^2(S_{21})}\\
 &&  + \,C\,  \| |\xi|^s\,|\xi_1|^{\sigma} e^{-C\, |\xi_1|^2 t}\, |\xi_1|^{-\sigma}\,\widehat{u_0}(\xi) \|_{L^2(S_{22})}\\
  &\le& C\, t^{-\frac{\sigma}{2}}\, \|u_0\|_{\mathring H^{s, -\sigma}}.
\eeno 
We now estimate $\||\xi|^s\,K_2(\xi,t) \widehat{\theta_0}(\xi)\|_{L^2(S_1)}$. 
Invoking (\ref{gb2}) in Proposition \ref{432} and proceeding as in (\ref{ty}), we have
\ben
\||\xi|^s\,K_2(\xi,t) \widehat{\theta_0}(\xi)\|_{L^2(S_1)} &\le& C\, t\, \||\xi|^s\,e^{-c_0 |\xi|^2 t}\,\widehat{\theta_0}(\xi)\|_{L^2(S_1)} \notag \\
&\le& C\, t^{-\frac12 (s+\sigma) +1}\, \|\Lambda_1^{-\sigma} \theta_0\|_{L^2(\mathbb R^2)}.
\label{hg}
\een
We now turn to $\||\xi|^s\,K_2(\xi,t) \widehat{\theta_0}(\xi)\|_{L^2(S_{2})}$. By 
(\ref{gb4}), 
\ben
&& \||\xi|^s\,K_2(\xi,t) \widehat{\theta_0}(\xi)\|_{L^2(S_{2})} \le C\, \| |\xi|^s\,\frac{\xi_1\xi_2}{|\xi|^4} e^{-c_0 |\xi|^2 t}\widehat{\theta_0}(\xi)\|_{L^2(S_{2})}
\notag\\
&& \qquad\qquad + C\, \Big\| |\xi|^s\, \frac{|\xi_1||\xi_2|}{|\xi|^4} \, e^{-c_0 \frac{\xi_1^2 \xi_2^2}{|\xi|^2} t}\, e^{-c_0\,\frac{\xi_1^2}{|\xi|^4} t}\,\widehat{\theta_0}(\xi)\Big\|_{L^2(S_{2})}.
\label{po}
\een
The first part in (\ref{po}) can be bounded as in (\ref{ty}) and (\ref{hg}), 
\beno
\| |\xi|^s\,\frac{\xi_1\xi_2}{|\xi|^4} e^{-c_0 |\xi|^2 t}\widehat{\theta_0}(\xi)\|_{L^2(S_{2})}
&\le& \| |\xi|^{s-2} \, e^{-c_0 |\xi|^2 t}\widehat{\theta_0}(\xi)\|_{L^2(\mathbb R^2)}\\
&\le& C\, t^{-\frac12 (s+\sigma) +1}\, \|\Lambda_1^{-\sigma} \theta_0\|_{L^2(\mathbb R^2)}.
\eeno
To estimate the second piece in (\ref{po}), we invoke the simple fact that $x^m \,e^{-x} \le C(m)$ valid for any $m\ge 0$ and $x\ge 0$, and proceed as in (\ref{hh11}) and (\ref{hh2}) to obtain 
\beno
|\xi|^s\, \frac{|\xi_1||\xi_2|}{|\xi|^4} \, e^{-c_0 \frac{\xi_1^2 \xi_2^2}{|\xi|^2} t}\, e^{-c_0\,\frac{\xi_1^2}{|\xi|^4} t} &=& |\xi|^{s-1} \frac{|\xi_2|}{|\xi|} \, e^{-c_0 \frac{\xi_1^2 \xi_2^2}{|\xi|^2} t}\,t^{-\frac12} \, \frac{|\xi_1| t^{\frac12}}{|\xi|^2} \,e^{-c_0\,\frac{\xi_1^2}{|\xi|^4} t}\\
&\le& C\, t^{-\frac12}\, |\xi|^{s-1} e^{-c_0 \frac{\xi_1^2 \xi_2^2}{|\xi|^2} t}\\
&\le& \left\{\begin{array}{ll} C\, t^{-\frac12}\, |\xi|^{s-1}\, e^{-C \xi_2^2 t}\quad \mbox{for $\xi\in S_{21}$,}\\ C\, t^{-\frac12}\, |\xi|^{s-1}\, e^{-C \xi_1^2 t}\quad \mbox{for $\xi\in S_{22}$.} \end{array}\right.
\eeno
Therefore, the second term in (\ref{po}) can be bounded by 
\beno 
&& \Big\| |\xi|^s\, \frac{|\xi_1||\xi_2|}{|\xi|^4} \, e^{-c_0 \frac{\xi_1^2 \xi_2^2}{|\xi|^2} t}\, e^{-c_0\,\frac{\xi_1^2}{|\xi|^4} t}\,\widehat{\theta_0}(\xi)\Big\|_{L^2(S_{2})} \\
&& \le \Big\| |\xi|^s\, \frac{|\xi_1||\xi_2|}{|\xi|^4} \, e^{-c_0 \frac{\xi_1^2 \xi_2^2}{|\xi|^2} t}\, e^{-c_0\,\frac{\xi_1^2}{|\xi|^4} t}\,\widehat{\theta_0}(\xi)\Big\|_{L^2(S_{21})} \\
&&\quad  + \Big\| |\xi|^s\, \frac{|\xi_1||\xi_2|}{|\xi|^4} \, e^{-c_0 \frac{\xi_1^2 \xi_2^2}{|\xi|^2} t}\, e^{-c_0\,\frac{\xi_1^2}{|\xi|^4} t}\,\widehat{\theta_0}(\xi)\Big\|_{L^2(S_{22})} \\
&& \le C\, t^{-\frac12}\,\||\xi|^{s-1}\, e^{-C \xi_1^2 t}\,\widehat{\theta_0}(\xi)\|_{L^2} + C\, t^{-\frac12}\,\||\xi|^{s-1}\, e^{-C \xi_2^2 t}\,\widehat{\theta_0}(\xi)\|_{L^2} \\
&& \le C\, t^{-\frac12-\frac{\sigma}{2}}\, \|\theta_0\|_{\mathring H^{s-1, -\sigma}}.
\eeno 
We have completed the estimates of $\|u_1(t)\|_{\mathring H^s(\mathbb R^2)}$. Collecting the estimates yields 
\beno
\|u_1(t)\|_{\mathring H^s(\mathbb R^2)} &\le& C\, t^{-\frac12 (s+\sigma)}\, \|\Lambda_1^{-\sigma} u_{10}\|_{L^2(\mathbb R^2)} + C\, t^{-\frac{\sigma}{2}}\, \|u_{10}\|_{\mathring H^{s, -\sigma}(\mathbb R^2)}\\
&& + \,C\, t^{-\frac12 (s+\sigma) +1}\, \|\Lambda_1^{-\sigma} \theta_0\|_{L^2(\mathbb R^2)}
+ C\,t^{-\frac12-\frac{\sigma}{2}}\, \|\theta_0\|_{\mathring H^{s-1, -\sigma}(\mathbb R^2)}.
\eeno
$\|u_2(t)\|_{\mathring H^s(\mathbb R^2)}$ can be estimated very similarly. Only the last piece 
is bounded slightly differently.  Its upper bound is 
\beno
\|u_2(t)\|_{\mathring H^s(\mathbb R^2)} &\le& C\, t^{-\frac12 (s+\sigma)}\, \|\Lambda_1^{-\sigma} u_{20}\|_{L^2(\mathbb R^2)} + C\, t^{-\frac{\sigma}{2}}\, \|u_{20}\|_{\mathring H^{s, -\sigma}(\mathbb R^2)}\\
&& + \,C\, t^{-\frac12 (s+\sigma) +1}\, \|\Lambda_1^{-\sigma} \theta_0\|_{L^2(\mathbb R^2)}
+ C\,t^{-1-\frac{\sigma}{2}}\, \|\theta_0\|_{\mathring H^{s, -\sigma}(\mathbb R^2)}.
\eeno
The estimate of $\|\theta(t)\|_{\mathring H^s(\mathbb R^2)}$ is also similar, 
\beno
\|\theta(t)\|_{\mathring H^s(\mathbb R^2)} &\le& \,C\, t^{-\frac12 (s+\sigma) +1}\, \|\Lambda_1^{-\sigma} u_{20}\|_{L^2(\mathbb R^2)}
+ C\,t^{-\frac{\sigma}{2}}\, \|u_{20}\|_{\mathring H^{s-2, -\sigma}(\mathbb R^2)}\\
&& + \,C\, t^{-\frac12 (s+\sigma)}\, \|\Lambda_1^{-\sigma} \theta_0\|_{L^2(\mathbb R^2)} + C\, t^{-\frac{\sigma}{2}}\, \|\theta_0\|_{\mathring H^{s, -\sigma}(\mathbb R^2)}.
\eeno
This completes the proof of Theorem \ref{main1}.
\end{proof}

\vskip .3in 
\section{Proof of Theorem \ref{main2}}
\label{sec3}

This section proves Theorem \ref{main2}. The proof makes use of the wave structure in (\ref{nnn}) to construct 
a Lyapunov functional for the Fourier piece of the solution away from the axes in the frequency space. The construction involves a suitable combination of two energy inequalities.

\begin{proof}[Proof of Theorem \ref{main2}] Let $\widehat\varphi$ be the Fourier cutoff 
function defined in (\ref{phi}). Taking the convolution of $\varphi$ with the velocity 
equation in (\ref{nnn}) leads to 
\beq\label{sa}
\p_{tt} (\varphi*u) -( \eta \p_{11} + \nu \p_{22})  \p_t (\varphi*u) + \nu \eta \p_{11} \p_{22} (\varphi*u)  + \p_{11}\Delta^{-1} (\varphi*u) = 0.
\eeq
Dotting (\ref{sa}) with $\p_t (\varphi\ast u)$, we find 
\ben
&&\frac{1}{2} \frac{d}{dt} \left(\|\partial_{t} (\varphi*u)\|_{L^2}^2 +  \|\mathcal{R}_1 (\varphi*u)\|_{L^2}^{2}+ \eta \nu \|\partial_{12}(\varphi*u)\|_{L^{2}}^{2}\right) \notag\\
&& \qquad \qquad + \nu  \|\partial_{2}\partial_t (\varphi*u)\|_{L^{2}}^{2}+ \eta \|\partial_{1}\partial_t (\varphi*u)\|_{L^{2}}^{2}=0, \label{ex1}
\een
where we have written $\mathcal R_1 = \p_1 (-\Delta)^{-\frac12}$, the standard notation 
for the Riesz transform. Dotting (\ref{sa}) with $\varphi\ast u$ yields
\beno
&&\frac{1}{2} \frac{d}{dt}( \nu \|\partial_{2}(\varphi*u)\|_{L^{2}}^{2}+ \eta \|\partial_{1}(\varphi*u)\|_{L^{2}}^{2})+ \|\mathcal{R}_1 (\varphi*u)\|_{L^2}^{2}\\
&&\qquad  + \nu \eta \|\partial_{12}(\varphi*u)\|_{L^{2}}^{2}+ \int \partial_{tt} (\varphi*u)\cdot (\varphi*u) dx=0.
\eeno
Writing 
$$
\int \partial_{tt} (\varphi*u)\cdot (\varphi*u) dx = \frac{d}{dt}\int \partial_{t} (\varphi*u)\cdot (\varphi*u)\,dx  -\|\partial_{t}(\varphi*u)\|_{L^{2}}^{2},
$$
we obtain 
\ben
&&\frac{1}{2} \frac{d}{dt}\left(\nu \|\partial_{2}(\varphi*u)\|_{L^{2}}^{2}+ \eta \|\partial_{1}(\varphi*u)\|_{L^{2}}^{2} + 2 (\partial_{t} (\varphi*u), (\varphi*u))\right)\notag\\
&&\qquad+ \|\mathcal{R}_1 (\varphi*u)\|_{L^2}^{2}   + \nu \eta \|\partial_{12}(\varphi*u)\|_{L^{2}}^{2}- \|\partial_{t}(\varphi*u)\|_{L^{2}}^{2} =0. \label{ex2}
\een
where $(f, g)$ denotes the $L^2$-inner product. Let $\lambda>0$. Then (\ref{ex1}) $+ \lambda$ (\ref{ex2}) yields 
\ben\label{ss1}
 \frac{d}{dt} A(t) + 2 B(t) =0,
\een
where 
\beno
A(t) &:=& \|\partial_{t}(\varphi * u)\|_{L^{2}}^{2}+ \|\mathcal{R}_{1} (\varphi * u)\|_{L^{2}}^{2}+     \eta \nu \|\partial_{12}(\varphi * u) \|_{L^{2}}^{2}\\
 && +  \lambda \nu \|\partial_{2}(\varphi * u)\|_{L^{2}}^{2}+  \lambda \eta \|\partial_{1}(\varphi * u)\|_{L^{2}}^{2}
 + 2\lambda (\partial_t (\varphi * u), (\varphi * u)),\\
B(t) &:=& \nu \|\partial_2\partial_{t} (\varphi * u)\|_{L^{2}}^{2}
+\eta \|\partial_1\partial_{t} (\varphi * u)\|_{L^{2}}^{2}
+ \lambda \eta \nu \|\partial_{12}(\varphi * u)\|_{L^{2}}^{2}\\
&&  - \lambda \|\partial_{t} (\varphi * u)\|_{L^{2}}^{2}+ \lambda  \|\mathcal{R}_{1}(\varphi * u)\|_{L^{2}}^{2}.
\eeno
Our immediate goal here is to show that, if we choose $\lambda=\lambda(\nu, \eta, a_1, a_2)$ suitably, then there is a constant $C_0 = C_0(\nu, \eta, a_1, a_2)>0$ such that, for  any $t\ge 0$, 
\beq\label{comp}
B(t) \ge C_0\, A(t). 
\eeq
Recall that $a_1>0$ and $a_2>0$ are the parameters involved in the definition of the frequency cutoff function defined by (\ref{phi}). We now prove (\ref{comp}). By Plancherel's theorem, 
\beq\label{jk1}
\|\partial_2\partial_{t} (\varphi * u)\|_{L^{2}}^{2} = \int_{|\xi_1|\ge a_1, |\xi_2|\ge a_2} |\xi_2\,\p_t (\widehat\varphi  \widehat u(\xi,t))|^2\,d\xi \ge a_2^2\, \|\partial_{t} (\varphi * u)\|_{L^{2}}^{2}.
\eeq
Similarly, 
\ben
&& \|\partial_1\partial_{t} (\varphi * u)\|_{L^{2}}^{2} \ge a_1^2 \, \, \|\partial_{t} (\varphi * u)\|_{L^{2}}^{2}, \quad  \|\partial_{12}(\varphi * u)\|_{L^{2}}^{2} \ge a_1^2 \, \|\p_2 (\varphi * u)\|_{L^{2}}^{2},\label{jk2}\\
&& 
\|\partial_{12}(\varphi * u)\|_{L^{2}}^{2} \ge a_2^2 \, \|\p_1 (\varphi * u)\|_{L^{2}}^{2}, \quad \|\partial_{12}(\varphi * u)\|_{L^{2}}^{2} \ge a_1^2\,a_2^2 \, \|\varphi * u\|_{L^{2}}^{2}. 
\label{jk3}
\een
If $\lambda>0$ satisfies 
$$
\lambda \le \frac12 (\nu \,a_2^2 + \eta\, a_1^2), 
$$
then, by (\ref{jk1}), (\ref{jk2}) and (\ref{jk3}), 
\beno
B(t) &\ge& (\nu\, a_2^2 + \eta\, a_1^2)\|\partial_{t} (\varphi * u)\|_{L^{2}}^{2}- \lambda \|\partial_{t} (\varphi * u)\|_{L^{2}}^{2}+ \frac14 \lambda \eta \nu \|\partial_{12}(\varphi * u)\|_{L^{2}}^{2}\\
&&  + \frac14 \lambda \eta \nu\, a_1^2 \, \|\p_2 (\varphi * u)\|_{L^{2}}^{2} + \frac14 \lambda \eta \nu\, a_2^2 \, \|\p_1 (\varphi * u)\|_{L^{2}}^{2} \\
&& +\frac14\lambda \eta \nu \,a_1^2\, a_2^2\, \|\varphi*u\|_{L^2}^2  + \lambda  \|\mathcal{R}_{1}(\varphi * u)\|_{L^{2}}^{2}\\
&\ge& \frac12\,(\nu\, a_2^2 + \eta\, a_1^2)\|\partial_{t} (\varphi * u)\|_{L^{2}}^{2}+ \frac14 \lambda \eta \nu \|\partial_{12}(\varphi * u)\|_{L^{2}}^{2}\\
&& + \frac14 \lambda \eta \nu\, a_1^2 \, \|\p_2 (\varphi * u)\|_{L^{2}}^{2} + \frac14 \lambda \eta \nu\, a_2^2 \, \|\p_1 (\varphi * u)\|_{L^{2}}^{2}\\
&& +  \frac14\lambda \eta \nu \,a_1^2\, a_2^2\, \|\varphi*u\|_{L^2}^2+ \lambda  \|\mathcal{R}_{1}(\varphi * u)\|_{L^{2}}^{2}.
\eeno
By the Cauchy-Schwarz inequality, 
\beno 
&& \frac14\,(\nu\, a_2^2 + \eta\, a_1^2)\|\partial_{t} (\varphi * u)\|_{L^{2}}^{2} +  \frac14\lambda \eta \nu \,a_1^2\, a_2^2\, \|\varphi*u\|_{L^2}^2\\
&& \ge \frac12 \sqrt{\nu\, a_2^2 + \eta\, a_1^2} \,\sqrt{\lambda \eta \nu \,a_1^2\, a_2^2}\,
(\partial_{t} (\varphi * u), \varphi*u).
\eeno 
Therefore, 
\beno 
B(t) &\ge& \frac14\,(\nu\, a_2^2 + \eta\, a_1^2)\|\partial_{t} (\varphi * u)\|_{L^{2}}^{2} 
+ \lambda  \|\mathcal{R}_{1}(\varphi * u)\|_{L^{2}}^{2}\\
&& + \frac14 \lambda \eta \nu \|\partial_{12}(\varphi * u)\|_{L^{2}}^{2}  + \frac14 \lambda \eta \nu\, a_1^2 \, \|\p_2 (\varphi * u)\|_{L^{2}}^{2} + \frac14 \lambda \eta \nu\, a_2^2 \, \|\p_1 (\varphi * u)\|_{L^{2}}^{2}\\
&& + \frac12 \sqrt{\nu\, a_2^2 + \eta\, a_1^2} \,\sqrt{\lambda \eta \nu \,a_1^2\, a_2^2}\,
(\partial_{t} (\varphi * u), \varphi*u). 
\eeno 
If we choose $C_0$ as 
$$
C_0 = \frac14\min\left\{(\nu\, a_2^2 + \eta\, a_1^2),\, \lambda,\,   \eta a_1^2,\,  \nu a_2^2, \,\frac{1}{\sqrt{\lambda}}\,\sqrt{\nu\, a_2^2 + \eta\, a_1^2} \,\sqrt{\eta \nu \,a_1^2\, a_2^2}\, \right\},
$$
then $B(t) \ge C_0 A(t)$, which is (\ref{comp}). Inserting (\ref{comp}) in (\ref{ss1}) leads to 
\beq\label{imm}
A(t) \le A(0) e^{-C_0 t}.
\eeq
To prove (\ref{ss2}), we derive a lower bound for $A(t)$. By (\ref{jk3}) and the Cauchy-Schwarz inequality, 
\beno
A(t) &\ge& \|\partial_{t}(\varphi * u)\|_{L^{2}}^{2}+ \|\mathcal{R}_{1} (\varphi * u)\|_{L^{2}}^{2}+     \eta \nu\, a_1^2 \,a_2^2\, \|\varphi * u \|_{L^{2}}^{2} \notag\\
&& +  \lambda \nu \|\partial_{2}(\varphi * u)\|_{L^{2}}^{2}+  \lambda \eta \|\partial_{1}(\varphi * u)\|_{L^{2}}^{2}
-\frac12 \|\partial_{t}(\varphi * u)\|_{L^{2}}^{2} - 2 \lambda^2 \|\varphi * u\|_{L^{2}}^{2}
\notag\\
&=& \frac12 \|\partial_{t}(\varphi * u)\|_{L^{2}}^{2} +  \|\mathcal{R}_{1} (\varphi * u)\|_{L^{2}}^{2} + (\eta \nu\, a_1^2 \,a_2^2 - 2 \lambda^2) \|\varphi * u\|_{L^{2}}^{2}
\notag \\
&& + \lambda \nu \|\partial_{2}(\varphi * u)\|_{L^{2}}^{2}+  \lambda \eta \|\partial_{1}(\varphi * u)\|_{L^{2}}^{2}.
\eeno
If $\lambda$ is selected  to satisfy
$$
\eta \nu\, a_1^2 \,a_2^2 - 2 \lambda^2 \ge \frac12 \eta \nu\, a_1^2 \,a_2^2 
\quad \mbox{or}\quad \lambda \le \frac12\, \sqrt{\eta\,\nu} \,a_1 \,a_2, 
$$
then $A(t)$ is bounded below by 
\ben
A(t) &\ge&  \frac12 \|\partial_{t}(\varphi * u)\|_{L^{2}}^{2} +  \|\mathcal{R}_{1} (\varphi * u)\|_{L^{2}}^{2} + \frac12 \eta \nu\, a_1^2 \,a_2^2 \,\|\varphi * u\|_{L^{2}}^{2} \notag\\
&& + \lambda \nu \|\partial_{2}(\varphi * u)\|_{L^{2}}^{2}+  \lambda \eta \|\partial_{1}(\varphi * u)\|_{L^{2}}^{2} \notag \\
&\ge& C\, (\|\partial_{t}(\varphi * u)\|_{L^{2}}^{2} + \|\varphi * u\|_{L^{2}}^{2}
+ \|\na(\varphi * u)\|_{L^{2}}^{2}),  \label{lww}
\een
where $C=C(\nu, \eta, a_1, a_2)>0$ is a constant. 
We now derive an upper bound for $A(0)$. Recalling that $(u, \theta)$ satisfies 
\beno
&&\partial_t u_1=  \nu \partial_{22}u_1- \Delta ^{-1} \partial_{1}\partial_{2} \theta,  \\
&&\partial_t u_2= \nu \partial_{22}u_2+ \Delta ^{-1} \partial_{1}\partial_{1} \theta, \\
&& \p_t \theta = \eta\, \p_{11} \theta - u_2,  
\eeno
we obtain 
$$
\partial_t u_1(0) = \nu \partial_{22}u_{01}- \Delta ^{-1} \partial_{1}\partial_{2} \theta_0, \qquad \partial_t u_2(0) = \nu \partial_{22}u_{02}+ \Delta ^{-1} \partial_{1}\partial_{1} \theta_0
$$
and thus
\beq\label{gf1}
\|(\p_t (\phi * u)(0)\|_{L^2}^2 \le  2\nu^2 \|\p_{22} (\varphi *u_0)\|_{L^2}^2 + 2\|\varphi *\theta_0\|_{L^2}^2,
\eeq
where we have used the fact that Riesz transforms are bounded in $L^q$ with $1<q<\infty$ (see
\cite{Stein}), 
$$
\|\Delta ^{-1} \partial_{1}\partial_{2} \,f\|_{L^q} \le C\, \|f\|_{L^q}. 
$$
In addition, if we invoke the inequality 
$$
2\lambda (\partial_t (\varphi * u), (\varphi * u)) \le \|\p_t (\phi * u)\|_{L^2}^2 + \lambda^2
\|\varphi * u\|_{L^2}^2,
$$
we obtain the following upper bound for $A(0)$, 
\ben
A(0) &:=& \|\partial_{t}(\varphi * u)(0)\|_{L^{2}}^{2}+ \|\mathcal{R}_{1} (\varphi * u_0)\|_{L^{2}}^{2} + \eta \nu \|\partial_{12}(\varphi * u_0) \|_{L^{2}}^{2} \notag\\
&& +  \lambda \nu \|\partial_{2}(\varphi * u_0)\|_{L^{2}}^{2}+  \lambda \eta \|\partial_{1}(\varphi * u_0)\|_{L^{2}}^{2}
+ 2\lambda (\partial_t (\varphi * u)(0), (\varphi * u_0)), \notag\\
&\le& 4\nu^2 \|\p_{22}(\varphi * u_0)\|_{L^2}^2 + 4\|\varphi *\theta_0\|_{L^2}^2
+ (1+\lambda^2) \|\varphi *u_0\|_{L^2}^2 \notag\\
&& + \eta \nu \|\partial_{12}(\varphi * u_0) \|_{L^{2}}^{2}+  \lambda \nu \|\partial_{2}(\varphi * u_0)\|_{L^{2}}^{2}+  \lambda \eta \|\partial_{1}(\varphi * u_0)\|_{L^{2}}^{2} \notag\\
&\le& \,C\, (\|\varphi * u_0\|_{H^{2}}^{2} + \|\varphi *\theta_0\|_{L^2}^2). \label{upp}
\een 
Combining (\ref{imm}), (\ref{lww}) and (\ref{upp}), we find that 
\beno
&& \|\partial_{t}(\varphi * u)(t)\|_{L^{2}}^{2} + \|(\varphi * u)(t)\|_{L^{2}}^{2}
+ \|\na(\varphi * u)(t)\|_{L^{2}}^{2} \\
 && \qquad \le \,C\, (\|\varphi * u_0\|_{H^{2}}^{2} + \|\varphi *\theta_0\|_{L^2}^2)\, e^{-C_0 t},
\eeno
which is  (\ref{ss2}). The proof for the exponential decay upper bound for $\theta$ in (\ref{ss3})
is very similar. In fact, since $\theta$ satisfies the same wave equation as $u$, most of the lines for $u$ remain valid when we replace $u$ by $\theta$ and replace the bound 
in (\ref{gf1}) by 
$$
\|(\p_t (\phi * \theta)(0)\|_{L^2}^2 \le  2\eta^2 \|\p_{11} (\varphi *\theta_0)\|_{L^2}^2 + 2\|\varphi *u_{02}\|_{L^2}^2.
$$
This completes the proof of Theorem \ref{main2}. 
\end{proof}

\vskip .3in 
\section{Proof of Theorem \ref{main3}}
\label{sec4}

This section is devoted to the proof of Theorem \ref{main3}. As outlined in the introduction, 
the proof uses the bootstrapping argument and the major step is to establish the energy
inequality 
\beq\label{ene}
E(t) \le C_1\, E(0) + C_2\, E(t)^{\frac32},
\eeq
where $C_1$ and $C_2$ are constants and $E(t)$ is the energy functional defined in (\ref{ee}), or
\ben
E(t) &=& \max_{0\leq \tau \leq t} (\|u(\tau)\|_{H^2}^{2} +\|\theta(\tau)\|_{H^2}^{2})  + 2 \nu \int_{0}^{t} \|\partial_2 u\|_{H^2}^{2}d\tau \notag\\
&& + 2\eta \int_{0}^t \|\partial_1 \theta\|_{H^2}^{2} d\tau + \delta \int_{0}^{t}\|\partial_1 u_2\|_{L^2}^{2}\, d\tau, \label{ee1}
\een
with $\delta>0$ to be specified later. We then apply the bootstrapping argument to (\ref{ene}) to get the desired stability result. 

\vskip .1in 
\begin{proof}[Proof of Theorem \ref{main3}]
We define $E(t)$ as in (\ref{ee1}). Our main efforts are devoted to establishing (\ref{ene}). 
This process consists of two major parts. The first is to estimate the $H^2$-norm of $(u, \theta)$ while the second is to estimate $\|\p_1 u\|_{L^2}^2$ and its time integral. 

\vskip .1in 
For a divergence-free vector field $u$, namely $\na\cdot u=0$, we have 
$$
\|\na u\|_{L^2} = \|\om\|_{L^2}, \quad \|\Delta u\|_{L^2} = \|\na \om\|_{L^2},
$$
where $\om=\na\times u$ is the vorticity. Therefore, the $H^2$-norm of $u$ is equivalent to the 
sum of the $L^2$-norm of $u$, the $L^2$-norm of $\om$ and the  $L^2$-norm of $\na \om$. To estimate the $L^2$-norm of $(u, \theta)$, we take the inner product of $(u, \theta)$
with the first two equations in (\ref{bb1}) to obtain 
\ben
&&\|u(t)\|_{L^2}^{2} +\|\theta(t)\|_{L^2}^{2}  + 2 \nu \int_{0}^{t} \|\partial_2 u(\tau)\|_{L^2}^{2}\, d\tau
+ 2\eta \int_{0}^t \|\partial_1 \theta(\tau)\|_{L^2}^{2} d\tau \notag\\
&& = \|u_0\|_{L^2}^{2} +\|\theta_0\|_{L^2}^{2}.\label{l2estimate}
\een
To estimate the $L^2$-norm of $(\om, \na \theta)$, we resort to the vorticity equation 
combined with the equation of $\theta$, 
\begin{equation}\label{vorticityequation}
\begin{aligned}
&\partial_t \omega +u\cdot \nabla \omega=  \nu \partial_{22}\omega + \partial_1 \theta, \\
&\partial_t \theta + u \cdot \nabla \theta + u_2= \eta \partial_{11} \theta.
\end{aligned}
\end{equation}
Taking the inner product of $(\om, \na\theta)$ with the equations of $\om$ and $\na\theta$, we 
obtain 
\begin{equation}
\begin{aligned}
\frac{1}{2}\frac{d}{dt}(\|\omega\|_{L^2}^{2}+\|\nabla \theta\|_{L^2}^{2})+ \nu \|\partial_2 \omega \|_{L^2}^{2}+ \eta  \|\partial_1 \nabla \theta \|_{L^2}^{2}= I_1 + I_2, \label{coo}
\end{aligned}
\end{equation}
where 
$$
 I_1 =\int (\partial_1 \theta\, \omega- \nabla u_2\cdot \nabla \theta)\, dx, \quad I_2= -\int \nabla \theta \cdot \nabla u \cdot \nabla \theta\, dx.
$$
It is easy to check that 
$$
I_1 =0.
$$
In fact, writing $\om$ and $u$ in terms of the stream function $\psi$, namely $\om =\Delta \psi$ and $u= \na^\perp \psi:=(-\p_2 \psi, \p_1 \psi)$, we have 
\beno
I_1 &=& \int (\partial_1 \theta\, \omega- \nabla u_2\cdot \nabla \theta)\, dx 
= \int (\p_1 \theta \Delta \psi - \na \p_1 \psi\cdot \nabla \theta)\, dx\\
&=& \int (-\theta \, \Delta\p_1 \psi + \Delta \p_1 \psi \, \theta)\,dx =0. 
\eeno
To bound $I_2$, we write out the four terms in $I_2$ explicitly,
\beno
I_2 &=&  -\int (\partial_1 u_1 (\partial_1 \theta)^2+ \partial_1 u_2 \partial_1 \theta \partial_2 \theta +\partial_2 u_1 \partial_1 \theta \partial_2 \theta+ \partial_2 u_2 (\partial_2 \theta)^2)\,dx\\
&:=& I_{21} + I_{22} + I_{23} + I_{24}.
\eeno 
The terms on the right-hand side can be bounded as follows. The key point here is to obtain
upper bounds that are time integrable.  By Lemma \ref{tri}, 
\beno
|I_{21}| &\le& C\, \|\partial_1u_1\|_{L^2} \|\partial_1 \theta\|_{L^2}^{\frac{1}{2}} \|\partial_2 \partial_1 \theta \|_{L^2}^{\frac{1}{2}} \|\partial_1 \theta\|_{L^2}	   ^{\frac{1}{2}} \|\partial_1 \partial_1\theta \|_{L^2}^{\frac{1}{2}}\\
&\le& C\, \|\partial_1u_1\|_{L^2}\, \|\partial_1\theta \|_{L^2}\,\|\p_1 \na \theta\|_{L^2},
\eeno 
\beno
|I_{22}| &\le& C\,\|\partial_1 \theta \|_{L^2} \|\partial_1 u_2 \|_{L^2}^{\frac{1}{2}} \|\partial_2 \partial_1 u_2\|_{L^2}^{\frac{1}{2}} \|\partial_2 \theta\|_{L^2}^{\frac{1}{2}} \|\partial_1 \partial_2\theta\|_{L^2}^{\frac{1}{2}}\\
&\le& C\,\|\partial_1 u_2 \|_{L^2}^{\frac{1}{2}}\, \|\partial_2 \theta\|_{L^2}^{\frac{1}{2}}\,
\|\partial_1 \theta \|_{L^2} \,\|\p_2 \na u\|_{L^2}^{\frac12}\, \|\p_1\na \theta\|_{L^2}^{\frac12},
\eeno
\beno
|I_{23}| &\le& C\,\|\partial_2 \theta \|_{L^2} \|\partial_1 \theta\|_{L^2}^{\frac{1}{2}} \|\partial_2 \partial_1 \theta\|_{L^2}^{\frac{1}{2}} \|\partial_2 u_1\|_{L^2}^{\frac{1}{2}} \|\partial_1 \partial_2 u_1 \|_{L^2}^{\frac{1}{2}}.
\eeno
By the divergence-free condition $\na\cdot u=0$, 
\beno 
I_{24} &=& \int \partial_1 u_1 (\partial_2 \theta)^2\,  dx =  - 2\int u_1\, \partial_2 \theta\, \partial_1\partial_2 \theta\,dx\\
&\le&  C\,  \|\partial_1\partial_2 \theta\|_{L^2} \|\partial_2 \theta\|_{L^2}^{\frac{1}{2}} \|\partial_1 \partial_2 \theta \|_{L^2}^{\frac{1}{2}}  \| u_1\|_{L^2}^{\frac{1}{2}} \|\partial_2 u_1\|_{L^2}^{\frac{1}{2}} \\
&=& C\,  \| u_1\|_{L^2}^{\frac{1}{2}}\, \|\partial_2 \theta\|_{L^2}^{\frac{1}{2}}\, \|\partial_2 u_1\|_{L^2}^{\frac{1}{2}}\, \|\p_1\na\theta\|_{L^2}^{\frac{3}{2}}.
\eeno 
Clearly, the sum of the powers of the terms that contain the favorable derivatives ($\p_1$ on $\theta$ and $\p_2$ on $u$) is $2$ in each upper bound above. Therefore each upper bound is time integrable. Collecting the upper bounds on $I_2$ and inserting them in (\ref{coo}), we obtain 
\ben
&& \frac{d}{dt}(\|\nabla u\|_{L^2}^{2}+\|\nabla \theta\|_{L^2}^{2})+ 2\nu \|\partial_2 \nabla u \|_{L^2}^{2}+ 2\eta  \|\partial_1 \nabla \theta \|_{L^2}^{2} \notag\\
&& \qquad\qquad \leq C\, (\|u\|_{H^1} + \|\na \theta\|_{L^2})\, \left(\|\p_2 u\|_{H^1}^2 + \|\p_1\theta\|_{H^1}^2\right). \label{h1sum}
\een
Integrating  (\ref{h1sum}) over $[0, t]$ and combining with $(\ref{l2estimate})$, we obtain 
\ben
&& \|(u, \theta)\|_{H^1}^{2}+2 \nu  \int_{0}^{t}\|\partial_2  u(s) \|_{H^1}^{2}ds + 2 \eta  \int_{0}^{t}\|\partial_1  \theta(s) \|_{H^1}^{2}ds \notag\\
&& \leq \|(u_0, \theta_0)\|_{H^1}^{2} +  \,C\, \int_0^t (\|u\|_{H^1} + \|\na \theta\|_{L^2})\, \left(\|\p_2 u\|_{H^1}^2 + \|\p_1\theta\|_{H^1}^2\right)\,d\tau \label{hh1} \\
&& \le E(0) + \,C\, E(t)^{\frac{3}{2}}. \label{h1final}
\een
We also notice that the $H^1$-estimate is actually self-contained. The upper bound in (\ref{hh1})
depends only on the $H^1$-norm level quantities. A simple consequence of (\ref{hh1}) is
that any initial small $H^1$ initial data leads to a global $H^1$ weak solution. However, we do not know the uniqueness of $H^1$-level solutions. This is one of the reasons that we are seeking 
global $H^2$-solutions. 

\vskip .1in 
In order to control the $H^2$-norm, it then suffices to bound the $L^2$-norm of $(\na \om, \Delta \theta)$. Applying $\nabla$ to the first equation of (\ref{vorticityequation}) and dotting with $\nabla \omega$, and apply $\Delta$ to the second equation of (\ref{vorticityequation}) and dotting with $\Delta \theta$,  we obtain 
\beq\label{h2energy}
\frac{1}{2}\frac{d}{dt}(\|\nabla \omega\|_{L^2}^{2} +\|\Delta \theta(t)\|_{L^2}^{2})+  \nu \|\partial_2 \nabla \omega\|_{L^2}^{2}+ \eta \|\partial_1 \Delta \theta\|_{L^2}^{2} 
= J_1 + J_2 + J_3,
\eeq
where 
\beno 
&& J_1 = \int (\nabla \partial_1 \theta \cdot \nabla \omega- \Delta u_2 \Delta \theta)\,dx, \\
&& J_2 = -\int \nabla \omega \cdot \nabla u \cdot \nabla \omega\, dx, \\ 
&& J_3 = - \int \Delta \theta \cdot \Delta (u\cdot \nabla \theta) dx.
\eeno 
First we verify that $J_1 =0$. In fact, since $u_2 = \p_1 \psi$ and $\Delta \psi = \om$, we have 
\beno 
J_1 &=& \int (\nabla \partial_1 \theta \cdot \nabla \omega- \Delta u_2 \Delta \theta)\,dx
= \int (\nabla \partial_1 \theta \cdot \nabla \omega- \Delta \p_1 \psi \Delta \theta)\,dx\\
&=& \int (\nabla \partial_1 \theta \cdot \nabla \omega-  \p_1 \om\, \Delta \theta)\,dx
 = \int (\nabla \partial_1 \theta \cdot \nabla \omega + \p_1 \na\om\, \cdot\na \theta)\,dx\\
&=& \int \p_1(\na\theta \cdot \na \om)\,dx =0. 
\eeno 
We now estimate $J_3$ and then $J_2$. The effort is still devoted to obtaining an upper bound that is time integrable
for each term.  After integration by parts, 
\beno
J_3 &=& -\int \Delta \theta\,  \Delta u_1 \, \partial_1 \theta\, dx - \int \Delta \theta \, \Delta u_2\, \partial_2 \theta\, dx\\
&&-2\int \Delta \theta\, \nabla u_1\cdot \partial_1\nabla\theta\, dx -2\int \Delta \theta \,\nabla u_2  \cdot \partial_2 \nabla \theta\, dx\\
&:=& J_{31} + J_{32} + J_{33} + J_{34}.
\eeno
By Lemma \ref{tri}, 
\ben
|J_{31}| &\le& C\, \|\p_1 \theta\|_{L^2}\, \|\Delta\theta\|_{L^2}^{\frac12}\, \|\p_1\Delta\theta\|_{L^2}^{\frac12}\, \|\Delta u_1\|_{L^2}^{\frac12}\, \|\p_2\Delta u_1\|_{L^2}^{\frac12} \notag\\
&\le& \,C\, (\|\Delta\theta\|_{L^2} + \|\Delta u_1\|_{L^2})\, 
\|\p_1 \theta\|_{H^2}^{\frac32}\, \|\p_2\Delta u_1\|_{L^2}^{\frac12}. \label{j31b}
\een
The bound on the right-hand side is time integrable. To bound $J_{32}$, we further decompose it into two terms,
\begin{equation*}
\begin{aligned}
J_{32}=&-\int \Delta \theta  \Delta u_2\partial_2 \theta dx\\
=&-\int \partial_1 \partial_1 \theta \,\Delta u_2 \,\partial_2 \theta dx -\int \partial_2 \partial_2 \theta \, \Delta u_2\, \partial_2 \theta \, dx \\
=&-\int \partial_1 \partial_1 \theta\, \Delta u_2\partial_2\, \theta \,dx +\frac{1}{2}\int \Delta  \partial_2 u_2\, (\partial_2 \theta)^2\, dx \\
=&-\int \partial_1 \partial_1 \theta\, \Delta u_2\,\partial_2 \theta\, dx -\frac{1}{2}\int \Delta  \partial_1 u_1\, (\partial_2 \theta)^2 dx \\
=&-\int \partial_1 \partial_1 \theta\, \Delta u_2\,\partial_2 \theta\, dx +\int \Delta  u_1\, \partial_2 \theta\, \partial_1 \partial_2 \theta dx.
\end{aligned}
\end{equation*}
Therefore, by Lemma \ref{tri},
\ben
|J_{32}|
&\leq &\, C\,\|\partial_1 \partial_1 \theta\|_{L^2}\|\Delta u_2\|_{L^2}^{\frac{1}{2}}\|\partial_2 \Delta  u_2\|_{L^2}^{\frac{1}{2}} \|\partial_2 \theta\|_{L^2}^{\frac{1}{2}} \|\partial_1 \partial_2 \theta\|_{L^2}^{\frac{1}{2}} \notag\\
&& +C\|\partial_1 \partial_2 \theta\|_{L^2} \|\partial_2 \theta\|_{L^2}^{\frac{1}{2}} \|\partial_1 \partial_2 \theta\|_{L^2}^{\frac{1}{2}} \|\Delta u_1\|_{L^2}^{\frac{1}{2}}\|\partial_2 \Delta  u_1\|_{L^2}^{\frac{1}{2}} \notag\\
&\leq & \,C\, (\|\p_2 \theta\|_{L^2} + \|\Delta u\|_{L^2})\, \|\p_1\na \theta\|_{L^2}^{\frac32}\,\|\p_2\Delta u\|_{L^2}^{\frac12}. \label{j32b}
\een
$J_{33}$ can be bounded as follows,
\ben
|J_{33}| &\le&\,C\, \|\p_1\na\theta\|_{L^2}\, \|\Delta\theta\|_{L^2}^{\frac12}\, \|\p_1\Delta\theta\|_{L^2}^{\frac12}\, \|\na u_1\|_{L^2}^{\frac12}\,\|\p_2\na u_1\|_{L^2}^{\frac12} \notag\\
&\le& \,C\, (\|\Delta\theta\|_{L^2} + \|\na u_1\|_{L^2})\, 
\|\p_1 \theta\|_{H^2}^{\frac32}\, \|\p_2\na u_1\|_{L^2}^{\frac12}.\label{j33b}
\een
By integration by parts, 
\begin{equation*}
\begin{aligned}
J_{34} 
& = -2 \int (\partial_1 u_2 \partial_1 \partial_2 \theta \Delta \theta+ \partial_2 u_2 \partial_2 \partial_2 \theta \Delta \theta)\, dx \\
& = -2\int \partial_1 u_2\, \partial_1\partial_2 \theta\, \Delta \theta\, dx + 2 \int \partial_1 u_1\, \partial_2 \partial_2 \theta\, \Delta \theta\, dx \\
&= -2\int \partial_1 u_2\, \partial_1\partial_2 \theta\, \Delta \theta\, dx -2 \int  u_1\, \partial_1 \partial_2 \partial_2 \theta\, \Delta \theta\, dx  -2 \int  u_1\,  \partial_2 \partial_2 \theta\, \partial_1 \Delta \theta\, dx\\
&:= J_{341} + J_{342} + J_{343}.
\end{aligned}
\end{equation*}
The terms on the right can be bounded as follows. 
\begin{equation*}
\begin{aligned}
|J_{341}|
\leq&\, C\,\|\partial_1\partial_2 \theta\|_{L^2} \|\partial_1 u_2\|_{L^2}^{\frac{1}{2}}\|\partial_2\partial_1 u_2\|_{L^2}^{\frac{1}{2}} \|\Delta \theta\|_{L^2}^{\frac{1}{2}}\|\partial_1 \Delta  \theta\|_{L^2}^{\frac{1}{2}}\\
\leq&\, C\,(\|\Delta\theta\|_{L^2} + \|\p_1 u_2\|_{L^2})\, 
\|\p_1 \theta\|_{H^2}^{\frac32}\, \|\p_2\na u_2\|_{L^2}^{\frac12},
\end{aligned}
\end{equation*}
\begin{equation*}
\begin{aligned}
|J_{342}|\leq& C\|\partial_1\partial_2 \partial_2 \theta\|_{L^2} \|\Delta \theta\|_{L^2}^{\frac{1}{2}}\|\partial_1 \Delta  \theta\|_{L^2}^{\frac{1}{2}} \| u_1\|_{L^2}^{\frac{1}{2}}\|\partial_2 u_1\|_{L^2}^{\frac{1}{2}} \\
\leq&\, C\,(\|\Delta\theta\|_{L^2} + \|u_1\|_{L^2})\, 
\|\p_1 \theta\|_{H^2}^{\frac32}\, \|\p_2 u_1\|_{L^2}^{\frac12},
\end{aligned}
\end{equation*}
\begin{equation*}
\begin{aligned}
|J_{343}|\leq& C\,\|\partial_1\Delta \theta\|_{L^2} \|\partial_2 \partial_2 \theta\|_{L^2}^{\frac{1}{2}}\|\partial_1 \partial_2 \partial_2 \theta\|_{L^2}^{\frac{1}{2}} \| u_1\|_{L^2}^{\frac{1}{2}}\|\partial_2 u_1\|_{L^2}^{\frac{1}{2}} \\
\leq& C\, (\|\Delta\theta\|_{L^2} + \|u_1\|_{L^2})\, 
\|\p_1 \theta\|_{H^2}^{\frac32}\, \|\p_2 u_1\|_{L^2}^{\frac12}.
\end{aligned}
\end{equation*}
Combining these estimates yields
\begin{equation} \label{j34b}
|J_{34}| \leq C\, (\|\theta\|_{H^2} + \|u\|_{H^2})\, \|\p_1 \theta\|_{H^2}^{\frac32}\, \|\p_2 u\|_{H^2}^{\frac12}.
\end{equation}
Putting  (\ref{j31b}), (\ref{j32b}), (\ref{j33b}) and (\ref{j34b}) together, we obtain 
\beq\label{j3b}
|J_3| \le \,C\, (\|\theta\|_{H^2} + \|u\|_{H^2})\, \|\p_1 \theta\|_{H^2}^{\frac32}\, \|\p_2 u\|_{H^2}^{\frac12}.
\eeq
We now turn to the estimate of $J_2$. As we have explained in the introduction, we need the help of the extra regularization term 
\beq\label{qqq}
\int_0^t \|\p_1 u_2\|_{L^2}^2\,d\tau.
\eeq
To make full use of the anisotropic dissipation, we further write $J_2$ as 
\begin{equation*}
\begin{aligned}
J_2 =&-\int \partial_1 u_1\,  (\partial_1\omega)^2 \, dx - \int \partial_1 u_2\,  \partial_1 \omega \, \partial_2 \omega\,  dx \\
& - \int \partial_2 u_1\, \partial_1 \omega \, \partial_2 \omega dx-\int \partial_2 u_2\,  (\partial_2\omega)^2 dx\\
=&\int \partial_2 u_2\,  (\partial_1\omega)^2 dx - \int \partial_1 u_2\,  \partial_1 \omega\,  \partial_2 \omega \, dx \\
& - \int \partial_2 u_1\,  \partial_1 \omega\,  \partial_2 \omega \, dx - \int \partial_2 u_2\,  (\partial_2\omega)^2\,  dx\\
:=& J_{21} + J_{22} + J_{23} + J_{24}. 
\end{aligned}
\end{equation*}
To bound the first two terms, we need to make use of the term in (\ref{qqq}). By integration by parts and Lemma \ref{tri},
\beno 
J_{21} &=& -2\int  u_2 \,\partial_1\omega \,\p_2\partial_{1}\omega\, dx \\
&\leq & C\|\partial_2\partial_1 \omega\|_{L^2} \|\partial_1 \omega\|_{L^2}^{\frac{1}{2}}\|\partial_2 \partial_1 \omega\|_{L^2}^{\frac{1}{2}} 
\|u_2\|_{L^2}^{\frac{1}{2}}\|\partial_1 u_2\|_{L^2}^{\frac{1}{2}} \\
&\le& \,C\, (\|u_2\|_{L^2} + \|\p_1\om\|_{L^2})\, \|\p_2\p_1\om\|_{L^2}^{\frac32}\, \|\p_1 u_2\|_{L^2}^{\frac12}.
\eeno 
By Lemma \ref{tri},
\beno
|J_{22}|
&\leq & \,C\,\|\partial_1u_2\|_{L^2} \|\partial_1 \omega\|_{L^2}^{\frac{1}{2}}\|\partial_2 \partial_1 \omega\|_{L^2}^{\frac{1}{2}} \| \partial_2 \omega\|_{L^2}^{\frac{1}{2}}\|\partial_1 \partial_2 \omega\|_{L^2}^{\frac{1}{2}} \\
&\leq & \,C\, \|\na \om\|_{L^2} \, \|\p_2\p_1\om\|_{L^2}\, \|\p_1 u_2\|_{L^2},
\eeno
\beno 
|J_{23}| &\le& C\, \|\p_2 u_1\|_{L^2}\, \|\p_1\om\|_{L^2}^{\frac12}\, \|\p_2\p_1\om\|_{L^2}^{\frac12}\, \|\p_2 \om\|_{L^2}^{\frac12}\,\|\p_1\p_2 \om\|_{L^2}^{\frac12}\\
&\le& \,C\, \|\na \om\|_{L^2} \, \|\p_2\p_1\om\|_{L^2}\, \|\p_2 u_1\|_{L^2},
\eeno 
\beno
|J_{24}| &\le& C\, \|\p_2 u_2\|_{L^2}\, \|\p_2\om\|_{L^2}^{\frac12}\, \|\p_1\p_2\om\|_{L^2}^{\frac12}\, \|\p_2 \om\|_{L^2}^{\frac12}\,\|\p_2\p_2 \om\|_{L^2}^{\frac12}\\
&\le& \,C\, \|\na \om\|_{L^2} \, \|\p_2\na\om\|_{L^2}\, \|\p_2 u_2\|_{L^2}.
\eeno
Therefore,
\beq\label{j2b}
|J_2| \le \,C\, \|u\|_{H^2}\, (\|\p_2 \na \om \|_{L^2}^2 + \|\p_1 u_2\|_{L^2}^2 + \|\p_2 u_1\|_{L^2}^2).
\eeq
Inserting $J_1=0$, (\ref{j3b}) and (\ref{j2b}) in (\ref{h2energy}), we obtain 
\ben
&& \frac{d}{dt}(\|\Delta u\|_{L^2}^{2}+\|\Delta \theta\|_{L^2}^{2}) + 2\nu \|\partial_2 \Delta u \|_{L^2}^{2}+ 2\eta  \|\partial_1 \Delta \theta \|_{L^2}^{2}\notag \\
&& \leq \,C\, (\|\theta\|_{H^2} + \|u\|_{H^2})\, \|\p_1 \theta\|_{H^2}^{\frac32}\, \|\p_2 u\|_{H^2}^{\frac12}\notag\\
&&\quad  + \,C\, \|u\|_{H^2}\, (\|\p_2 \na \om \|_{L^2}^2 + \|\p_1 u_2\|_{L^2}^2 + \|\p_2 u_1\|_{L^2}^2) .\label{honesum}
\een
Integrating  (\ref{honesum}) over the time  interval $[0, t]$  yields 
\ben
&& \|\Delta u(t)\|_{L^2}^{2}+\|\Delta \theta(t)\|_{L^2}^{2} +2 \nu  \int_{0}^{t}\|\partial_2  \Delta u \|_{L^2}^{2}d\tau + 2 \eta  \int_{0}^{t}\|\Delta \partial_1  \theta \|_{L^2}^{2}d\tau
\notag \\
&& \le \|\Delta u_0\|_{L^2}^{2}+\|\Delta \theta_0\|_{L^2}^{2} + C\, \int_0^t  (\|\theta\|_{H^2} + \|u\|_{H^2})\, \|\p_1 \theta\|_{H^2}^{\frac32}\, \|\p_2 u\|_{H^2}^{\frac12}\,d\tau
\notag \\
&&\quad + C\, \int_0^t \|u\|_{H^2}\, (\|\p_2 \na \om \|_{L^2}^2 + \|\p_1 u_2\|_{L^2}^2 + \|\p_2 u_1\|_{L^2}^2) \,d\tau \notag \\
&& \leq E(0)  + \, C\,  E(t)^{\frac{3}{2}}. \label{H2final}
\een

\vskip .1in 
The next major step is to bound the last piece in $E(t)$ defined by (\ref{ee1}), namely
$$
\int_0^t \|\p_1 u_2\|_{L^2}^2\,d\tau.
$$
We make use of the equation of $\theta$. By the equation of $\theta$,
\begin{equation}\label{p1u}
\partial_1 u_2= -\partial_t \partial_1 \theta - \partial_1 (u \cdot \nabla \theta) + \eta \partial_{111} \theta. 
\end{equation}
Multiplying (\ref{p1u}) with $\partial_1 u_2$ and then integrating over $\mathbb R^2$ yields
\beno
\|\partial_1 u_2\|_{L^2}^{2}&=& -\int \partial_t \partial_1 \theta\,  \partial_1 u_2\;dx - \int \partial_1 u_2\, \partial_1 (u \cdot \nabla \theta)\;dx + \eta \int \partial_1 u_2\,\partial_{111} \theta\;dx \\
&:=& K_1 +K_2 + K_3.
\eeno
Even though the estimate of $K_3$ appears to be easy, the term with unfavorable derivative $\p_1 u_2$ will be absorbed by the left-hand side, 
\beq\label{k1up}
|K_3|\le \,\eta \|\p_1 u_2\|_{L^2}\, \|\p_{111} \theta\|_{L^2} \le \frac12 \|\p_1 u_2\|_{L^2}^2 + C\, \|\p_1\theta\|_{H^2}^2. 
\eeq
We shift the time derivative in $K_1$, 
\beq\label{k3up}
K_1 = - \frac{d}{dt} \int \partial_1 \theta\, \partial_1 u_2\, dx+\int  \partial_1 \theta\, \partial_1 \partial_t u_2\, dx := K_{11} + K_{12}.
\eeq
Invoking the equation for the second component of the velocity, we have 
\beno
K_{12} &=& -\int \p_1\p_1\theta \, \p_t u_2\,dx \\
&=& -\int \partial_{11}  \theta  (-(u\cdot \nabla)u_2- \partial_2 p+\nu \partial_{22}u_2+ \theta)\;dx\\
&= & \int \partial_{11}  \theta\, (u\cdot \nabla) u_2\;dx\;+\int \partial_{11}  \theta \;\partial_2 p\;dx\\
&& - \nu \int \partial_{11}  \theta\; \partial_{22}u_2\;dx\; - \int \partial_{11}  \theta\; \theta\;dx.
\eeno
We further replace the pressure term. Applying the divergence operator to the velocity equation yields 
\begin{equation*}
p= -\Delta^{-1}\nabla \cdot (u\cdot \nabla u)+ \Delta^{-1} \partial_2 \theta.
\end{equation*}
Therefore, 
\beno
K_{12} &=& \int \partial_{11}  \theta\, (u\cdot \nabla) u_2\;dx\;+\int \partial_{11}  \theta\, (-\partial_2 \Delta^{-1}\nabla \cdot (u\cdot \nabla u))\;dx\\ 
&&-\nu \int \partial_{11}  \theta\; \partial_{22}u_2\;dx\; -\int \partial_{11}  \theta\; \partial_{11} \Delta^{-1}\theta\;dx\\
&:=& K_{121}+ K_{122}+ K_{123}+K_{124}.
\eeno
By the boundedness of the double Riesz transform (see, e.g., \cite{Stein}),
$$
\|\p_{11} \Delta^{-1} f\|_{L^q} \le \, C\, \|f\|_{L^q}, \qquad 1<q<\infty,
$$
we have 
$$
K_{124} = \int \p_1\theta\, \partial_{11} \Delta^{-1} \p_1 \theta\;dx \le C\, \|\p_1\theta\|_{L^2}^2. 
$$
$K_{123}$ can be easily bounded, 
$$
|K_{123}| \le C\, \|\p_{11}\theta\|_{L^2} \, \|\p_{22} u_2\|_{L^2}.
$$
By integration by parts and the boundedness of the double Riesz transform, 
\beno
K_{122} &=& -\int \partial_{1} \theta\; \partial_{12} \Delta^{-1}\nabla \cdot (u\cdot \nabla u)\;dx\\
&\leq & \|\partial_{1} \theta\|_{L^2}\; \| \Delta^{-1}\partial_{12} \nabla \cdot (u\cdot \nabla u)\|_{L^2}\\
&\leq & \,C\, \|\partial_{1} \theta\|_{L^2}\,\|\partial_{2} (u\cdot \nabla u)\|_{L^2} \\  
&\leq & \,C\, \|\partial_{1} \theta\|_{L^2}\,\|\partial_{2} u\cdot \nabla u+ u\cdot \nabla \partial_2 u\|_{L^2}  \\
&\leq & \,C\, \|\partial_{1} \theta\|_{L^2} \left(\;\|\partial_{2}u\|_{L^4} \;\| \nabla u\|_{L^4}+ \|u\|_{\infty}\|\nabla \partial_2 u\|_{L^2}\right) \\             
&\leq & \,C\,\|\partial_{1} \theta\|_{L^2} \,\|\partial_{2}u\|_{H^1}\;  \| \nabla u\|_{H^1}+
\,C\,\|\partial_{1} \theta\|_{L^2} \, \|u\|_{H^2}\|\nabla \partial_2 u\|_{L^2}.
\eeno 
To bound $K_{121}$, we further split it, 
\beno
K_{121} &=& \int \partial_{11}  \theta (u_1 \partial_1 u_2+ u_2 \partial_2 u_2) dx\\
&=&\int \partial_{11}  \theta\; u_1\; \partial_1 u_2 \;dx\;+ \int \partial_{11}\theta\; u_2\; \partial_2 u_2\; dx.
\eeno
By Lemma \ref{tri}, 
\beno
|K_{121}| &\le& \, C\, \|\partial_{11} \theta\|_{L^2}\,\|u_1\|_{L^2}^{\frac{1}{2}}\|\partial_1 u_1\|_{L^2}^{\frac{1}{2}} \|\partial_1 u_2\|_{L^2}^{\frac{1}{2}}\|\partial_2 \partial_1 u_2\|_{L^2}^{\frac{1}{2}}\\
&& + \,C\,\|u_2\|_{L^\infty}\, \|\p_{11}\theta\|_{L^2}\, \|\p_2 u_2\|_{L^2}\\
&\le& \,C\, \|u\|_{H^1}\, \|\p_2u\|_{H^1}\, \|\partial_{11} \theta\|_{L^2} +\,C\, 
\|u\|_{H^2}\, \|\p_2 u\|_{L^2}\,\|\p_{11}\theta\|_{L^2}.
\eeno
We have thus obtained an upper bound for $K_{12}$,
\beq \label{k12up}
|K_{12}| \le \,C\, \|\p_1\theta\|_{L^2}^2 + C\, \|\p_{11}\theta\|_{L^2} \, \|\p_{22} u_2\|_{L^2}
+\,C\, \|u\|_{H^2}\, \|\p_2 u\|_{H^1}\,\|\p_1\theta\|_{H^1}.
\eeq
It remains to bound $K_2$. We decompose $K_2$ into four terms,
\beno
K_2 &=&  - \int \partial_1 u_2\, \partial_1 u_1\, \partial_1  \theta\,  dx - \int \partial_1 u_2  u_1 \partial_1 \partial_1  \theta\, dx\\
&&- \int \partial_1 u_2 \partial_1 u_2 \partial_2  \theta\,  dx - \int \partial_1 u_2  u_2 \partial_1 \partial_2  \theta\, dx. 
\eeno 
By Lemma \ref{tri}, 
\ben
|K_2| &\le& \,C\,  \|\partial_1u_2\|_{L^2} \|\partial_2 u_2\|_{L^2}^{\frac{1}{2}}\|\partial_2 \partial_2 u_2\|_{L^2}^{\frac{1}{2}} \|\partial_1 \theta\|_{L^2}^{\frac{1}{2}} \|\partial_1 \partial_1 \theta\|_{L^2}^{\frac{1}{2}} \notag \\
&& + \,C\,  \|u_1\|_{L^2}^{\frac{1}{2}}\|\partial_1 u_1\|_{L^2}^{\frac{1}{2}} \|\partial_1 u_2\|_{L^2}^{\frac{1}{2}} \|\partial_2 \partial_1 u_2\|_{L^2}^{\frac{1}{2}}\|\partial_1 \partial_1 \theta\|_{L^2} \notag\\
&& +  \,C\, \|\partial_1 u_2\|_{L^2} \| \partial_1 u_2\|_{L^2}^{\frac{1}{2}}\|\partial_2 \partial_1 u_2\|_{L^2}^{\frac{1}{2}} \|\partial_2 \theta\|_{L^2}^{\frac{1}{2}} \|\partial_2 \partial_1 \theta\|_{L^2}^{\frac{1}{2}} \notag\\
&& +  \,C\,\|\partial_1 \partial_2 \theta\|_{L^2} \|u_2\|_{L^2}^{\frac{1}{2}}\|\partial_1 u_2\|_{L^2}^{\frac{1}{2}} \|\partial_1 u_2\|_{L^2}^{\frac{1}{2}} \|\partial_2 \partial_1 u_2\|_{L^2}^{\frac{1}{2}} \notag\\
&\le& \,C\, \|u\|_{H^1}\, (\|\p_2 u\|_{H^1}^2 + \|\p_1\theta\|_{H^1}^2) \notag\\
&& + \,C\,(\|u\|_{H^2} +
\,\|\theta\|_{H^2}) (\|\p_1 u_2\|_{L^2}^2 + \|\p_1\theta\|_{H^1}^2). \label{k2up}
\een 
Combining (\ref{k1up}), (\ref{k3up}), (\ref{k12up}) and (\ref{k2up}), we find 
\beno 
\frac12 \, \|\p_1 u_2\|_{L^2}^2 &\le& \,C\, \|\p_1\theta\|_{H^2}^2 -\frac{d}{dt}\int \p_1\theta\, \p_1 u_2\,dx \\
&& + C\, \|\p_{11}\theta\|_{L^2} \, \|\p_{22} u_2\|_{L^2} + \,C\, \|u\|_{H^2}\, (\|\p_2 u\|_{H^1}^2 + \|\p_1\theta\|_{H^1}^2) \\
&& + \,C\,(\|u\|_{H^2} +
\,\|\theta\|_{H^2}) (\|\p_1 u_2\|_{L^2}^2 + \|\p_1\theta\|_{H^1}^2).
\eeno 
Integrating over $[0, t]$ yields 
\ben
\int_0^t \|\p_1 u_2\|_{L^2}^2 \,d\tau &\le& \,C\, \int_0^t \|\p_1\theta\|_{H^2}^2 \,d\tau
-2\int \p_1\theta\, \p_1 u_2\,dx + 2\int \p_1\theta_0\, \p_1 u_{02}\,dx \notag \\
&& + \,C\,\int_0^t  \|\p_{11}\theta\|_{L^2} \, \|\p_{22} u_2\|_{L^2}\,d\tau \notag\\
&& + \,C\,\int_0^t  \|u\|_{H^2}\, (\|\p_2 u\|_{H^1}^2 + \|\p_1\theta\|_{H^1}^2) \,d\tau \notag\\
&& + \,C\,\int_0^t (\|u\|_{H^2} +
\,\|\theta\|_{H^2}) (\|\p_1 u_2\|_{L^2}^2 + \|\p_1\theta\|_{H^1}^2)\,d\tau\notag\\
&\le& \,C\, \int_0^t \|\p_1\theta\|_{H^2}^2 \,d\tau + \,C\, \int_0^t \|\p_2 u\|_{H^2}^2\,d\tau
 + \,C\, (\|u\|_{H^1}^2 + \|\theta\|_{H^1}^2)\notag \\
&& + \,C\, (\|u_0\|_{H^1}^2 + \|\theta_0\|_{H^1}^2) + C\, E(t)^{\frac32}. \label{defb}
\een
We then combine the $H^1$-bound in (\ref{h1final}), the homogeneous $H^2$-bound in (\ref{H2final})
and the bound for the extra regularization term in (\ref{defb}). We need to eliminate 
the quadratic terms on the right-hand side of (\ref{defb}) by the corresponding terms 
on the left-hand side, so we need to multiply both sides of (\ref{defb}) by a suitable small coefficient
$\delta$. $(\ref{h1final}) + (\ref{H2final}) + \delta\,(\ref{defb})$ gives
\ben
&& \|u(t)\|_{H^2}^{2} +\|\theta(t)\|_{H^2}^{2} + 2 \nu \int_{0}^{t} \|\partial_2 u\|_{H^2}^{2}d\tau  
+ 2\eta \int_{0}^t \|\partial_1 \theta\|_{H^2}^{2} d\tau+ \delta \int_{0}^{t}\|\partial_1 u_2\|_{L^2}^{2} \notag \\
&&\leq \,  E(0) + \,C\,E(t)^{\frac{3}{2}}\;+\, C\,\delta\, (\|u(t)\|_{H^{2}}^{2} +\|\theta(t)\|_{H^{2}}^{2}) + \,C\,\delta\, (\|u_0\|_{H^{2}}^{2} + \|\theta_0\|_{H^{2}}^{2}) \notag\\
&&\quad  + \,C\,\delta \int_{0}^{t}\|\partial_2 u \|_{H^2}^{2} d\tau 
+ \,C\,\delta\int_{0}^{t}\|\partial_1 \theta  \|_{H^2}^{2} d\tau 
+ \,C\,\delta\, E(t)^{\frac{3}{2}}. \label{gir}
\een
If  $\delta>0$ is chosen to be sufficiently small, say 
$$
C\, \delta \le \frac12, \quad C\, \delta \le \nu, \quad C\, \delta \le \eta,
$$
then (\ref{gir}) is reduced to 
\begin{equation}\label{ggd}
E(t)\leq \,C_1\,  E(0)+ \, C_2\, E(t)^{\frac{3}{2}},
\end{equation}
where $C_1$ and $C_2$ are positive constants. An application of the bootstrapping argument to 
(\ref{ggd}) then leads to the desired stability result. In fact, if the initial data $(u_0, \theta_0)$ is sufficiently small, 
$$
\|(u_0, \theta_0)\|_{H^2} \le \varepsilon := \frac1{4 \sqrt{C_1}{C_2}},
$$
then (\ref{ggd}) allows us to show that 
$$
\|(u(t), \theta(t))\|_{H^2} \le \sqrt{2 C_1}\, \varepsilon.
$$
The bootstrapping argument starts with the ansatz that, for $t<T$
\beq\label{lov1}
E(t) \le \frac1{4 C_2^2}
\eeq
and show that
\beq\label{lov}
E(t) \le \frac1{8 C_2^2} \quad\mbox{for all $t\le T$}.
\eeq
Then the bootstrapping argument would imply that $T=\infty$ and (\ref{lov}) actually holds for all $t$. (\ref{lov}) is an easy consequence of (\ref{ggd}) and (\ref{lov1}). Inserting (\ref{lov1}) in (\ref{ggd}) yields 
\beno
E(t) &\leq& \,C_1\,  E(0)+ \, C_2\, E(t)^{\frac{3}{2}} 
\\
&\le& C_1 \, \varepsilon^2 + \, C_2\, \frac1{2C_2}\, E(t).
\eeno
That is, 
$$
\frac12 E(t) \le C_1 \, \varepsilon^2\quad\mbox{or}\quad E(t) \le 2\, C_1 \, \frac{1}{16 C_1 C_2^2} = \frac{1}{8 C_2^2} = 2C_1\, \epsilon^2,
$$
which is (\ref{lov}). This establishes the global stability.

\vskip .1in 
Finally we briefly explain the uniqueness. It is not difficult to see that the solutions to (\ref{bb1}) at this regularity level must be unique. Assume 
that $(u^{(1)}, p^{(1)}, \theta^{(1)})$ and $(u^{(2)}, p^{(2)}, \theta^{(2)})$ are two solutions of (\ref{bb1}) with one of them in the $H^2$-regularity class say $(u^{(1)}, \theta^{(1)})\in L^{\infty}(0, T; H^2) $.  The difference $(\widetilde{u}, \widetilde{p}, \widetilde{\theta})$ with 
$$
\widetilde{u}=u^{(2)}- u^{(1)},\quad \widetilde{p}=p^{(2)}- p^{(1)} \quad \text{and} \quad \widetilde{\theta}=\theta^{(2)}- \theta^{(1)} 
$$
satisfies 
\begin{equation}\label{uniqueness}
\begin{aligned}
&\partial_t \widetilde{u}+u^{(2)}\cdot \nabla \widetilde{u}+ \widetilde{u} \cdot \nabla u^{(1)}+\nabla \widetilde{p}= \nu \partial_{22}\widetilde{u}+ \widetilde{\theta}{\mathbf e}_2, \\
&\partial_t \widetilde{\theta} + u^{(2)} \cdot \nabla \widetilde{\theta} + \widetilde{u}\cdot \nabla \theta^{(1)} + \widetilde{u_2}\;= \eta \partial_{11} \widetilde{\theta}, \\
&\nabla \cdot \widetilde{u}=0,\\
&\widetilde u(x,0) =0, \quad \widetilde\theta(x,0) =0.
\end{aligned}
\end{equation}
We estimate the difference $(\widetilde{u}, \widetilde{p}, \widetilde{\theta})$ in $L^{2}(\mathbb{R}^2)$. Dotting (\ref{uniqueness}) by $(\widetilde{u}, \widetilde{\theta})$ and applying the divergence free condition, we find
$$
\frac{1}{2}\frac{d}{dt}\|(\widetilde{u}, \widetilde{\theta})\|_{L^2}^{2}+ \nu \|\partial_{2}\widetilde{u}\|_{L^2}^{2}+ \eta \|\partial_{1}\widetilde{\theta}\|_{L^2}^{2}= -\int \widetilde{u} \cdot \nabla u^{(1)}\cdot \widetilde{u}\; dx - \int \widetilde{u} \cdot \nabla \theta^{(1)}\cdot \widetilde{\theta}\; dx.
$$
By Lemma \ref{tri}, Young's inequality and the uniformly global bound for $\|(u^{(1)}, \theta^{(1)})\|_{H^2}$, we have
\beno
&& \frac{1}{2}\frac{d}{dt}\|(\widetilde{u}, \widetilde{\theta})\|_{L^2}^{2}+ \nu \|\partial_{2}\widetilde{u}\|_{L^2}^{2}+ \eta \|\partial_{1}\widetilde{\theta}\|_{L^2}^{2} \\
&&  \leq C\, \|\widetilde{u}\|_{L^2}\|\widetilde{u}\|_{L^2}^{\frac{1}{2}}\|\partial_2 \widetilde{u}\|_{L^2}^{\frac{1}{2}}\|\nabla u^{(1)}\|_{L^2}^{\frac{1}{2}}\|\partial_1 \nabla u^{(1)}\|_{L^2}^{\frac{1}{2}}\\
&& + \,C\, \|\widetilde{\theta}\|_{L^2}\|\widetilde{u}\|_{L^2}^{\frac{1}{2}}\|\partial_2 \widetilde{u}\|_{L^2}^{\frac{1}{2}}\|\nabla \theta^{(1)}\|_{L^2}^{\frac{1}{2}}\|\partial_1 \nabla \theta^{(1)}\|_{L^2}^{\frac{1}{2}}\\
&& \leq C\, \|\widetilde{u}\|_{L^2}^{\frac{3}{2}}\|\partial_2 \widetilde{u}\|_{L^2}^{\frac{1}{2}}+ C\, \|\widetilde{u}\|_{L^2}^{\frac{1}{2}}\|\partial_2 \widetilde{u}\|_{L^2}^{\frac{1}{2}}\|\widetilde{\theta}\|_{L^2}\\
&&\leq \frac{\nu}{2}  \|\p_2 \widetilde{u}\|_{L^2}^{2}+ C\,\|(\widetilde{u}, \widetilde{\theta})\|_{L^{2}}^{2}.
\eeno
It then follows from Gronwall's inequality that
$$
\|\widetilde{u}(t)\|_{L^2}= \|\widetilde{\theta}(t)\|_{L^{2}} =0.
$$
That is, these two solutions coincide.
This completes the proof of Theorem \ref{main3}. 
\end{proof}

\vskip .3in
\section*{\bf Acknowledgments}

This work was partially supported by the National Science Foundation of USA under grant DMS 1624146. Wu was partially supported the AT\&T Foundation at Oklahoma State University.

\end{document}